\documentclass{article}
\usepackage[final,nonatbib]{nips_2016}
\usepackage{amsmath,amssymb}
\usepackage{color}
\usepackage{caption}
\usepackage{subcaption}
\usepackage{footnote}
\makesavenoteenv{tabular}
\makesavenoteenv{table}
\usepackage{mathrsfs}
\usepackage{enumerate}
\usepackage{graphicx}
\usepackage{algorithm}
\usepackage{algpseudocode}

\usepackage[utf8]{inputenc} 
\usepackage[T1]{fontenc}    
\usepackage{url}            
\usepackage{booktabs}       
\usepackage{amsfonts}       
\usepackage{nicefrac}       
\usepackage{microtype}      

\newtheorem{theorem}{Theorem}[section]
\newtheorem{lemma}[theorem]{Lemma}
\newtheorem{corollary}[theorem]{Corollary}
\newtheorem{proposition}[theorem]{Proposition}

\newcommand{\citet}{\cite}

\newcommand{\lb}{\left(}
\newcommand{\rb}{\right)}
\newcommand{\eps}{\epsilon}

\newcommand{\td}{\tilde}
\newcommand{\E}{\mathbb{E}}
\newcommand{\R}{\mathbb{R}}
\newcommand{\I}{\mathcal{I}}

\newcommand{\pd}[2]{\frac{\partial #1}{\partial #2}}

\newcommand{\breg}[3]{#1(#2) - #1(#3) - \langle\nabla #1(#3), #2 - #3\rangle}
\newcommand{\ik}[3]{\langle\nabla #1(#3), #2 - #3 \rangle}

\newcommand{\la}{\langle}
\newcommand{\ra}{\rangle}
\renewcommand{\H}{\mathcal{H}}
\def \endprf{\hfill {\vrule height6pt width6pt depth0pt}\medskip}
\newenvironment{proof}{\noindent {\bf Proof} }{\endprf\par}

\newcommand{\xj}{x^{(j)}}
\newcommand{\yj}{y^{(j)}}
\newcommand{\nuj}{\nu^{(j)}}

\newcommand{\zetaj}{\zeta^{(j)}}

\newcommand{\etaj}{\eta}
\newcommand{\Bj}{B}

\newcommand{\ej}{e_{j}}
\newcommand{\Ej}{\E_{j}}
\newcommand{\Ij}{\I_{j}}
\newcommand{\Nj}{N_{j}}

\newcommand{\comp}{\mathrm{Comp}}

\DeclareMathOperator*{\argmin}{arg\,min}

\DeclareMathOperator*{\Var}{\ensuremath{Var}}
\DeclareMathOperator*{\Cov}{\ensuremath{Cov}}

\allowdisplaybreaks
\title{Less than a Single Pass: Stochastically Controlled Stochastic Gradient \thanks{We correct mistakes in the earlier version of the paper; See footnote 4 (p.~\pageref{algo:SCSG}) and footnote 18 (p.~\pageref{sec:ana_SCSG}), both in red,  for details.}}

\author{
  Lihua Lei \\
  Department of Statistics\\
  University of California, Berkeley\\
  Berkeley, CA 94704 \\
  \texttt{lihua.lei@berkeley.edu} \\
  \And
  Michael I. Jordan \\
  Computer Science Division \& Department of Statistics \\
  University of California, Berkeley\\
  Berkeley, CA 94704\\
  \texttt{jordan@stat.berkeley.edu} \\
}

\begin{document}

\maketitle


\begin{abstract}
We develop and analyze a procedure for gradient-based optimization that we refer to as \emph{stochastically controlled stochastic gradient} (SCSG).  As a member of the SVRG family of algorithms, SCSG makes use of gradient estimates at two scales and the number of updates is governed by a geometric random variable. Unlike most existing algorithms in this family, both the computation cost and the communication cost of SCSG do not necessarily scale linearly with the sample size $n$; indeed, these costs are independent of $n$ when the target accuracy is low. The experimental evaluation on real datasets confirms the effectiveness of SCSG.
\end{abstract}

\section{Introduction}
Optimizing the finite-sum convex objectives is ubiquitous in different areas: 
\begin{equation}
  \label{eq:obj_generic}
  \min_{x\in \R^{d}}f(x) = \frac{1}{n}\sum_{i=1}^{n}f_{i}(x),
\end{equation}
where  each $f_{i}(x)$ is a convex function. These problems are often solved by algorithms
that either make use of full gradients (obtained by processing the entire 
dataset) or stochastic gradients (obtained by processing single data points or mini-batches of data points).  The use of the former provides guarantees of eventual convergence and the latter yields advantages in terms of rate of convergence rate, scalability and simplicity of implementation~\cite{Hazan07, Nemirovsky09, Rakhlin12}.  An impactful recent line of research has shown that a hybrid methodology that makes use of both full gradients and stochastic gradients can obtain the best of both worlds---guaranteed convergence at favorable rates, e.g.~\cite{katyusha, SAGA, SVRG, Catalyst, SDCA}.  The full gradients provide variance control for the stochastic gradients.

While this line of research represents significant progress towards the goal of designing scalable, autonomous learning algorithms, there remain some inefficiencies in terms of computation. With the definition of computation and communication cost in Section 2.1, the methods referred to above require $O(n\cdot C(\eps, d))$ computation to achieve an $\eps$-approximate solution, where $n$ is the number of data points, $\eps$ is a target accuracy and $d$ is the dimension of the parameter vector.  Some methods incur a $O(nd)$ storage cost~\cite{SAGA, SAG}. The linear dependence on $n$ is problematic in general.  Clearly there will be situations in which accurate solutions can be obtained with less than a single pass through the data; indeed, some problems will require a constant number of steps.  This will be the case, for example, if the data in a regression problem consist of a fixed number of pairs repeated a large number of times. For deterministic algorithms, the worst case analysis in \cite{Agarwal14} shows that scanning at least a fixed proportion of the data is necessary;  however, learning algorithms are generally stochastic and real-world learning problems are generally not worst case. 

An equally important bottleneck for learning algorithms is the cost of communication. For large data sets that must be stored on disk or distributed across many computing nodes, the communication cost can be significant, even dominating the computation cost. For example, SVRG makes use of full gradient over the whole dataset which can incur prohibitive communication cost. There is an active line of research that focuses on communication costs; see, e.g.~\cite{arjevani15, CoCoA, konecny15, zhang12}.

In this article, we present a variant of the stochastic variance reduced gradient (SVRG) method that we refer to as \emph{stochastically controlled stochastic gradient} (SCSG). The basic idea behind SCSG---that of approximating the full gradient in SVRG via a subsample---has been explored by others, but we present several innovations that yield significant improvements both in theory and in practice. In contradistinction to SVRG, the theoretical convergence rate of SCSG has a sublinear regime in terms of both computation and communication.  This regime is important in machine learning problems, notably in the common situation in which the sample size is large, ($n \in [10^{4},10^{9}]$), while the required accuracy is low, $\eps\in [10^{-4},10^{-2}]$. The analysis in this article shows that SCSG is able to achieve the target accuracy in this regime with potentially less than a single pass through the data. 

In the regime of low accuracy, SCSG is never worse than the classical \emph{stochastic gradient descent} (SGD). Although SCSG has the same dependence on the target accuracy as SGD, it has a potentially much smaller factor. In fact, the theoretical complexity of SGD depends on the uniform bound of $\nabla f_{i}(x)$ over the domain and the component index. This might be infinite even in the most common least square problems. By contrast, the complexity of SCSG depends on a new measure $\H(f)$, defined in Section \ref{sec:notation} and discussed in Section \ref{sec:Hf}, which is finite and small for a large class of practical problems. In particular, $\H(f) = O(1)$ in many cases where SGD does not have theoretical guarantees to converge. The measure $\H(f)$ sheds light upon characterizing the difficulty of optimization problems in the form of a finite sum and reveals some intrinsic difference between finite-sum optimization and stochastic approximation, which is considered by other relevant works; e.g., streaming SVRG \cite{SSVRG} and dynaSAGA\cite{daneshmand16}.

The remainder of the paper is organized as follows.  In Section 2, we 
review SVRG, discuss several of its variants and  we describe the SCSG 
algorithm.  We provide a theoretical convergence analysis in Section 3.   
In Section 4, we give a comprehensive discussion on the difficulty measure $\H(f)$. The empirical results on real datasets are presented in Section \ref{sec:experiment}. Finally, we conclude our work and discuss potential extensions in Section \ref{sec:discussion}. All technical proofs are relegated to the Appendices.

\section{Notation, Assumptions and Algorithm}\label{sec:notation}
We write $\min\{a, b\}$ as $a\wedge b$ and $\max\{a, b\}$ as $a\vee b$ for brevity and use $\|\cdot\|$ to denote the Euclidean norm throughout the paper. We adopt the standard Landau's notation ($O(\cdot), o(\cdot), O_{p}(\cdot), o_{p}(\cdot), \Omega(\cdot)$). In some cases, we use $\td{O}(\cdot)$ to hide terms which are polynomial in parameters. The notation $\td{O}$ will only be used to maximize the readibility in discussions but not be used in the formal analysis. For convenience, we use $[n]$ to denote the set $\{1, \ldots, n\}$ and for any subset $\I\subset [n]$, we write $\nabla f_{\I}(x)$ the batch gradient $\frac{1}{|\I|}\sum_{i\in \I}\nabla f_{i}(x)$ for short. Finally, given random variables $Y$ and $Z$ and a random variable $X = f(Y, Z)$, denote by $\E_{X} Y$ the conditional expectation of $Y$ given $Z$, i.e. $\E(Y | Z)$. Note that when $Y$ is independent of $Z$, then $\E_{X}Y$ is equivalent to the the expectation of $Y$ holding $Z$ fixed. Furthermore, we use the symbol $\E$, without the subscript, to denote the expectation over all randomness. 

The assumption \textbf{A}1 on the smoothness of individual functions will be used throughout this paper. 
\begin{enumerate}[\textbf{A}1]
\item $f_{i}$ is convex with $L$-Lipschitz gradient
\[ f_{i}(x) - f_{i}(y) - \ik{f_{i}}{x}{y} \le \frac{L}{2}\|x - y\|^{2},
\]
for some $L < \infty$ and all $i\in \{1, \ldots, n\}$;
\end{enumerate}
The following assumption will be used in the context of strongly-convex objectives.
\begin{enumerate}[\textbf{A}2]
\item $f$ is strongly-convex with 
\[ f(x) - f(y) - \ik{f}{x}{y} \ge \frac{\mu}{2}\|x - y\|^{2},
\]
for some $\mu > 0$.
\end{enumerate}
Note that we only require the strong convexity of $f$ instead of each component. 

Let $x^{*}$ denote the minimizer of $f$ that minimizes $\H(f)$ in \eqref{eq:Hf}, then $\H(f)$ can be written as 
\begin{equation}
  \label{eq:Hf}
  \H(f) = \argmin_{x^{*}\in \argmin f(x)}\frac{1}{n}\sum_{i=1}^{n}\|\nabla f_{i}(x^{*})\|^{2}.
\end{equation}
We will abbreviate $\H(f)$ as $\H$ when no confusion can arise.  Note that $x^{*}$ is unique in many situations where $d < n$. When there are multiple minimum, we select $x^{*}$ be the one that minimizes the RHS of \eqref{eq:Hf}. Further let $\td{x}_{0}$ denote the initial value (possibly random) and 
\begin{equation}
  \label{eq:D0F0}
  \Delta_{x} = \E \|\td{x}_{0} - x^{*}\|^{2}, \quad \Delta_{f} = \E (f(\td{x}_{0}) - f(x^{*})).
\end{equation}
Then $\frac{\mu}{2}\Delta_{x}\le \Delta_{f}\le \frac{L}{2}\Delta_{x}$ under assumption \textbf{A}1 and  \textbf{A}2. A point $y$, possibly random, is called an $\eps$-approximated solution if
\[\E (f(y) - f(x^{*})) \le \eps.\]

In terms of the computation complexity, we assume that sampling an index $i$ and computing the pair $(f_{i}(x), \nabla f_{i}(x))$ incurs a unit of cost. This is conventional and called IFO framework in literature (\citet{Agarwal14, reddi16svrg}). 
We use  use $\comp(\eps)$ to denote the cost to achieve an $\eps$-accurate solution. In some contexts we also consider $\comp_{x}(\eps)$ as the cost to reach a solution $y$ with $\E \|y - x^{*}\|^{2}\le \eps$ \footnote{We only consider this quantity in the strongly convex case in which $x^{*}$ is uniquely defined.}. 

Finally, since our analysis heavily relies on geometric distributions, we formally define them here. We say a random variable $N\sim \mathrm{Geom}(\gamma)$ if $N$ is supported on non-negative integers \footnote{Here we allow $N$ to be zero to facilitate the analysis.} with
\[P(N = k) = (1 - \gamma)\gamma^{k},\quad \forall k = 0, 1, \ldots\]
The expectation of the above distributions satisfy that 
\begin{equation}\label{eq:exp_N}
\E N = \frac{\gamma}{1 - \gamma}.
\end{equation}



\begin{algorithm}[h]
\caption{Stochastic Variance Reduced Gradient (SVRG) Method}
\label{algo:svrg}
\textbf{Inputs: } Stepsize $\eta$, number of stages $T$, initial iterate $\td{x}_{0}$, number of SGD steps $m$.

\textbf{Procedure}
\begin{algorithmic}[1]
  \For{$j = 1, 2, \cdots, T$}
  \State $g_{j}\gets \nabla f(\td{x}_{j - 1}) = \frac{1}{n}\sum_{i = 1}^{n}\nabla f_{i}(\td{x}_{j - 1})$
  \State $\xj_{0}\gets \td{x}_{j - 1}$
  \State $N_{j}\gets m$
  \For{$k = 1, 2, \cdots, N_{j}$}
  \State Randomly pick $i_{k}\in [n]$
  \State $\nuj_{k - 1} \gets \nabla f_{i_{k}}(\xj_{k-1}) - \nabla f_{i_{k}}(\xj_{0}) + g_{j}$
  \State $\xj_{k}\gets \xj_{k-1} - \eta\nuj_{k - 1}$
  \EndFor
  \State $\td{x}_{j}\gets \xj_{N_{j}}$
  \EndFor
\end{algorithmic}
\textbf{Output: } (Option 1): $\td{x}_{T}$ \quad (Option 2): $\bar{x}_{T} = \frac{1}{T}\sum_{j=1}^{T}\td{x}_{j}$.
\end{algorithm}

\subsection{SVRG and Other Related Works}
The stochastic variance reduced gradient (SVRG) method blends gradient descent and stochastic gradient descent, using the former to control the effect of the variance of the latter~\cite{SVRG}.  We summarize SVRG in Algorithm \ref{algo:svrg}.

Using the definition from Section 2.1, it is easy to see that the computation cost of SVRG is $O((n + m)T)$. 
As shown in the convergence analysis of \cite{SVRG}, $m$ is required to be $\Omega(\kappa)$ to guarantee convergence. Thus, the computation cost of SVRG is $O((n + \kappa)T)$. The costs of the other algorithms considered in Table~\ref{tab:comp} can be obtain in a similar fashion. For comparison, we only present the results for smooth case (assumption \textbf{A}1).

A number of variants of SVRG have been studied.  For example, a constrained form of SVRG can be obtained by replacing line 8 with a projected gradient descent step~\cite{proximal_SVRG}.  A mini-batch variant of SVRG arises when one samples a subset of indices instead of a single index in line 6 and updates the iterates by the average gradient in this batch in line 7~\cite{mini_batch_SVRG}. Similarly, we can consider implementing the full gradient computation in line 2 using a subsample. This is proposed in \cite{PSVRG}, which calculates $g_{j}$ as $\frac{1}{B}\sum_{i\in \I}\nabla f_{i}(\td{x})$ where $\I$ is a subset of size $B$ uniformly sampled from  $\{1, \ldots, n\}$. \cite{PSVRG} heuristically show the potential for significant complexity reduction, but they only prove convergence for $B = \Omega(n)$ under the stringent condition that $\|\nabla f_{i}(x)\|$ is uniformly bounded for all $x$ and that all iterates are uniformly bounded. Similar to Nesterov's acceleration for gradient descent, momentum terms can be added to the SGD steps to accelerate SVRG~\cite{katyusha, AMSVRG}.

Much of this work focuses on the strongly convex case.  In the non-strongly convex setting one way to proceed is to add a $L_{2}$ regularizer 
$\frac{\lambda}{2}\|x\|^{2}$. Tuning $\lambda$, however, is subtle and 
requires multiple runs of the algorithm on a grid of $\lambda$~\cite{Zhu15}.  For general convex functions an alternative approach has been presented by \cite{Zhu15} (they generate $N_{j}$ by a different scheme in line 4), which proves a computation complexity $O\lb \frac{n}{\eps}\rb$. Another approach is discussed by \cite{reddi16svrg}, who improve the complexity to $O\lb n + \frac{\sqrt{n}}{\eps}\rb$ by scaling the stepsize as $O\lb \frac{1}{\sqrt{n}}\rb$. However, their algorithm still relies on calculating a full gradient. Other variants of SVRG have been proposed in the distributed computing setting~\cite{Lee15, Reddi15} and in the stochastic setting~\cite{daneshmand16, SSVRG}.

\subsection{SCSG}

\begin{algorithm}[htp]
\caption{Stochastically Controlled Stochastic Gradient (SCSG) Method} 
\label{algo:SCSG}
\textbf{Inputs: } Stepsize $\eta$, batch size $B$, number of stages $T$, initial iterate $\td{x}_{0}$.

\textbf{Procedure}
\begin{algorithmic}[1]
  \For{$j = 1, 2, \cdots, T$}
  \State Uniformly sample a batch $\I_{j}\in\{1, \cdots, n\}$ with $|\I_{j}| = B$
  \State $g_{j}\gets \frac{1}{B}\sum_{i\in\I_{j}}\nabla f_{i}(\td{x}_{j-1})$
  \State $\xj_{0}\gets \td{x}_{j-1}$
  \State Generate $N_{j}\sim \mathrm{Geom}\lb\frac{\Bj}{\Bj + 1}\rb$ 
  \For{$k = 1, 2, \cdots, N_{j}$}
  \State Randomly pick $i_{k}\in [n]$ 

  \State $\nuj_{k - 1} \gets \nabla f_{i_{k}}(\xj_{k-1}) - \nabla f_{i_{k}}(\xj_{0}) + g_{j}$
  \State $\xj_{k}\gets \xj_{k-1} - \eta\nuj_{k - 1}$
  \EndFor
  \State $\td{x}_{j}\gets \xj_{N_{j}}$
  \EndFor
\end{algorithmic}
\textbf{Output: } (Strongly convex case): $\td{x}_{T}$ \quad (Non-strongly convex case): $\bar{x}_{T} = \frac{1}{T}\sum_{j=1}^{T}\td{x}_{j}$.
\end{algorithm}

SCSG \footnote{{\color{red}The earlier paper samples $i_{k}$ from $\mathcal{I}_{j}$ in line 7 of Algorithm \ref{algo:SCSG}.}}is similar to \cite{PSVRG} in that it implements the gradient computation on a subsample $\I$ of size $B$; See Algorithm \ref{algo:SCSG}. However, instead of being fixed, the number of SGD updates of SCSG is a geometrically distributed random variable (line 5). Surprisingly, this seemingly technical modification enables the analysis in the non-strongly convex case and a much tighter convergence analysis without imposing unrealistic assumptions like the boundedness of iterates produced by the algorithm; See Section \ref{sec:ana} for details. Recently we found that \cite{hofmann15} also implicitly uses the geometric size of the inner loop. However, they do not use the iterate at the end of each epoch, i.e. $\td{x}_{j}$ and hence cannot prove the non-strongly convex case. 



\begin{table}[htp]
\interfootnotelinepenalty=10000
  \centering
  \centerline{\begin{tabular}{cccccc}
    \toprule
    & Non-Strongly Convex  & Strongly Convex  & Const. $\eta$?  & Dep. $\Delta_{x}, T$? & $f_{i}$ Lip.?\\
    \midrule
    SCSG & $O\lb \frac{1}{\eps^{2}}\wedge \frac{n}{\eps}\rb$ & $O\lb (\frac{\kappa}{\eps}\wedge n + \kappa)\log\frac{1}{\eps }\rb$ & Yes & No & No\\ 
    SGD\cite{reddi16svrg, Rakhlin12}\footnote{Theorem 6 of \cite{reddi16svrg} for the non-strongly convex case and Theorem 5 of \cite{Rakhlin12} for the strongly convex case.} & $O\lb\frac{1}{\eps ^{2}}\rb$ & $O\lb\frac{1}{\mu\eps}\rb$  & Yes / No & No / Yes\footnote{In the non-strongly convex case, the stepsize is either set to be $\Delta_{x} / \sqrt{T}$ for given number of total steps $T$. In the strongly convex case, the stepsize is set to be $\frac{c}{\mu t}$.} & Yes\\
    SVRG\cite{SVRG}\footnote{No result for the non-strongly convex case and Theorem 1 of \cite{SVRG} for the strongly convex case.} & - & $O\lb (n +  \kappa)\log\frac{1}{\eps }\rb$ & Yes & No & No\\
    MSVRG\cite{reddi16svrg}\footnote{Corollary 13 of \cite{reddi16svrg} for the non-strongly convex case and no result for the strongly convex case. The complexity bound $O\lb \frac{1}{\eps^{2}}\wedge \frac{\sqrt{n}}{\eps}\rb$ claimed in the paper is incorrect since it does not account for the cost of computing the full gradient.} & $O\lb n+\frac{1}{\eps^{2}}\wedge \frac{\sqrt{n}}{\eps} \rb$ & - & Yes & Yes & Yes\\
    SAGA\cite{SAGA}\footnote{Section 2 of \cite{SAGA} for both cases.} & $O\lb\frac{n}{\eps }\rb$ & $O\lb (n + \kappa)\log\frac{1}{\eps }\rb$ & Yes & No & No\\
    APSDCA\cite{APSDCA}\footnote{No result for non-strongly convex case and Theorem 1 of \cite{APSDCA} for the strongly convex case.} & - & $O\lb (n + \kappa)\log \frac{1}{\eps }\rb$ & Yes & No & No\\
    APCG\cite{APCG}\footnote{Theorem 1 of \cite{APCG} for both cases.} & $O\lb\frac{n}{\sqrt{\eps }}\rb$ & $O\lb \frac{n}{\sqrt{\mu}}\log\frac{1}{\eps }\rb$ & Yes & No & No\\
    SPDC\cite{SPDC}\footnote{No results for the non-strongly convex case and Section 1 of \cite{SPDC} for Empirical Risk Minimization.} & - & $O\lb (n + \sqrt{n\kappa})\log \frac{1}{\eps }\rb$ & Yes & No & No\\
    Catalyst\cite{Catalyst}\footnote{Table 1 of \cite{Catalyst} for both cases} & $O\lb\frac{n}{\sqrt{\eps }}\rb$ & $O\lb \lb n + \sqrt{n\kappa}\rb\log \frac{1}{\eps }\rb$ & No & Yes & No\\
    SVRG++\cite{Zhu15}\footnote{Theorem 4.1 of \cite{Zhu15} for the non-strongly convex case and no result for the strongly convex case.} & $O\lb n\log \frac{1}{\eps } + \frac{1}{\eps }\rb$ & - & Yes & Yes & No \\
    AMSVRG\cite{AMSVRG}\footnote{Theorem 2 of \cite{AMSVRG} for the non-strongly convex case and Theorem 3 of \cite{AMSVRG} for the strongly convex case.} & $O\lb\lb n + \frac{n}{\eps n + \sqrt{\eps }}\rb\log \frac{1}{\eps }\rb$ & $O\lb\lb n + \frac{n\kappa}{n + \sqrt{\kappa}}\rb\log \frac{1}{\eps }\rb$ & Yes & No & No\\
    Katyusha\cite{katyusha}\footnote{Corollary 4.3 of \cite{katyusha} for the non-strongly convex case and Theorem 3.1 for the strongly convex case.} & $O\lb n\log\frac{1}{\eps } + \sqrt{\frac{n}{\eps }}\rb$ & $O\lb (n + \sqrt{n\kappa})\log \frac{1}{\eps }\rb$ & No & No & No\\
    \bottomrule
  \end{tabular}}
  \caption{Comparison of the computation cost of SCSG and other algorithms for smooth convex objectives. The third column indicates whether the algorithm uses a fixed stepsize $\eta$; the fourth column indicates whether the tuning parameter depends on unknown quantities, e.g. $\Delta_{x}, \Delta_{f}, T$; the last column indicates whether $f_{i}$ is required to be Lipschitz or (almost) equivalently $\|\nabla f_{i}\|$ is required to be bounded.}\label{tab:comp}
\end{table}

The average computation cost of SCSG is $BT + \sum_{j=1}^{n}N_{j}$. By the law of large numbers and the expectation formula \eqref{eq:exp_N}, this is close to $2BT$. Table \ref{tab:comp} summarizes the computation complexity as well as some other details of SCSG and 11 other existing popular algorithms. The table includes the computation cost of optimizing non-strongly-convex functions (column 1) and strongly convex functions (column 2). In practice, the amount of tuning is of major concern. For this reason, a fixed stepsize is usually preferred to a complicated stepsize scheme and it is better that the tuning parameter does not depend on unknown quantities; e.g., $\Delta_{x}$ or the total number of epochs $T$. These issues are documented in column 3 and column 4. Moreover, many algorithms requires $\|\nabla f_{i}\|$ to be bounded, i.e. $f_{i}$ to be Lipschitz. However, this assumption is not realistic in many cases and it is better to discard it. To address this issue, we document it in column 5. To highlight the dependence on $\eps$ and $\kappa$ (or $\mu$), we implicitly assume that other parameters, e.g. $\Delta_{x}, L$, are $O(1)$ as a convention.

As seen from Table \ref{tab:comp}, SCSG and SGD are the only two methods which are able to reach an $\eps$-approximate solution with potentially less than a single pass through the data; moreover, the number of accesses of the data is independent of the sample size $n$. Comparing to SCSG, SGD requires each $f_{i}$ to be Lipschitz, which is not satisfied by least-square objectives. By contrast, as will be shown in Section \ref{sec:ana}, the computation cost of SCSG only depends on the quantity $\H(f)$, which is relatively small in many cases. Furthermore, SGD either sets the stepsize based on unknown quantities like the total number of epochs $T$ or needs to use a time-varying sequence of stepsizes. This involves intensive tuning as opposed to a fixed stepsize.

On the other hand, SCSG is communication-efficient since it only needs to operate on mini-batches as SGD. This is particularly important in modern large-scale tasks. By contrast, those algorithms that require full gradients evaluation either need extra communication for synchronization or need extra computational cost for the asynchronous version to converge; See e.g. \cite{Reddi15, Lee15}.

\section{Convergence Analysis}\label{sec:ana}
In this section we present a convergence analysis of SCSG. We first state the following key lemma that connects our algorithm with the measure $\H$ defined in \eqref{eq:Hf}.
\begin{lemma}\label{lem:sg_optim_var}
Let $\I\in\{1, \cdots, n\}$ be a random subset of size $B$, and 
define the random variable $g = \nabla f_{\I}(x^{*})$.
Then $\E g = 0$ and 
\[
\E \|g\|^{2} = \frac{(n - B) \H }{(n - 1)B}\le \frac{\H\cdot I(B < n)}{B}.
\]
\end{lemma}
The proof, which appears in Appendix \ref{app:one_epoch}, involves a standard technique for analyzing sampling without replacement.  Obviously, $ \H = O(1)$  if $\|\nabla f_{i}(x)\|$ is uniformly bounded as is often assumed in the literature. In section \ref{sec:Hf} we will present various other situations where $\H = O(1)$.

Note that the extra variation vanishes when $B = n$ and in general is inversely proportional to the batch size. In the rest of this section, we will first discuss the case $B = n$, which we refer to as R-SVRG (Randomized SVRG), to compare with the original SVRG. Later we will discuss the general case. 

\subsection{Analysis of R-SVRG}\label{sec:RSVRG}
We start from deriving the sub-optimality bound for $\bar{x}_{T}$ and $\td{x}_{T}$ respectively.

\begin{theorem}\label{thm:SVRG_bound}
Let $B = n$ and assume that $\eta L \le \frac{1}{3}$, then
\begin{enumerate}[(1)]
\item under the assumption \textbf{A}1,
\[\displaystyle \E (f(\bar{x}_{T}) - f(x^{*}))\le \frac{1}{T}\cdot\frac{4\etaj^{2}L  n  \Delta_{f} + \Delta_{x}}{2\etaj n (1 - 2\etaj L)};\]
\item under the assumption \textbf{A}1 and \textbf{A}2,
\[\displaystyle \E \|\td{x}_{T} - x^{*}\|^{2} + 2\etaj n \E (f(\td{x}_{T}) - f(x^{*}))\le \lambda^{T}\lb \Delta_{x} + 4\etaj n \Delta_{f}\rb,\] 
where 
\[\lambda = \max\left\{2\etaj L, \frac{1}{1 + \mu\etaj n(1 - 3\etaj L)}\right\}.\]
\end{enumerate}
\end{theorem}
Based on Theorem \ref{thm:SVRG_bound}, we first consider a constant stepsize $\etaj$ scaled as $\frac{1}{L}$. 

\begin{corollary}\label{cor:SVRG_complexity}
Let $\etaj = \frac{\theta}{L}$ with $\theta < \frac{1}{3}$. Then under the assumption \textbf{A}1, with the output $\bar{x}_{T}$, 
\begin{equation}\label{eq:SVRG_ns}
\E \comp(\eps) = O\lb \frac{n\Delta_{f} + L\Delta_{x}}{\eps}\rb.
\end{equation}
If further the assumption \textbf{A}2 is satisfied, then the output $\td{x}_{T}$ satisfies that 
\begin{equation}\label{eq:SVRG_sc}
\E \comp(\eps) = O\lb (n + \kappa) \log \lb\frac{\Delta_{f}}{\eps} + \frac{L\Delta_{x}}{n\eps}\rb\rb, \,\, \E \comp_{x}(\eps) = O\lb (n + \kappa) \log \lb\frac{n\Delta_{f}}{L\eps} + \frac{\Delta_{x}}{\eps}\rb\rb.
\end{equation}
\end{corollary}
The above theorem is appealing in three aspects: 1) in the strongly convex case, no parameter depends on $\mu$. This is in contrast to the original SVRG where the number of SGD updates should be proportional to $\kappa$ in order to guarantee the theoretical convergence \cite{SVRG}. \footnote{In \cite{SVRG}, the algorithm is guaranteed to converge only if $\frac{1}{\mu\etaj (1 - 2L\etaj)m} + \frac{2L\etaj}{1 - 2L\etaj} < 1$ where $m$ is the number of SGD updates. This entails that $m = \Omega(\kappa)$.} Being agnostic to $\mu$ is useful in that $\mu$ is hard to estimate in practice; 2) the same setup also guarantees the convergence of $\E \|\td{x}_{T} - x^{*}\|^{2}$ in the strongly convex case with an almost identical cost up to a $\log n$ factor. This is important especially in statistical problems but unfortunately not covered in existing literature to the best of our knowledge; 3) the same stepsize guarantees the convergence in both the non-strongly convex and the strongly convex case and the only requirement is $\eta < \frac{1}{3L}$, which is quite mild. Note that the requirement for the convergence of gradient descent is $\eta < \frac{1}{L}$.

By scaling $\etaj$ as $\frac{1}{\sqrt{n}}$, R-SVRG is able to achieve the same complexity of \cite{reddi16svrg}, which is the best bound in the class of SVRG-type algorithms without acceleration techniques.

\begin{corollary}\label{cor:SVRG_complexity_ns}
  Let $\etaj = \frac{\theta}{L\sqrt{n}}$ with $\theta \le \frac{1}{3}$. Then under the assumption \textbf{A}1, with the output $\bar{x}_{T}$,
\begin{equation}\label{eq:SVRG_ns_refine}
\E \comp(\eps) = O\lb n + \frac{\sqrt{n}L\Delta_{x}}{\eps}\rb.
\end{equation}
\end{corollary}

\subsection{Analysis of SCSG \protect\footnote{{\color{red} Our complexity bound $O\lb\frac{1}{\eps} + \kappa\rb$ in the earlier version for the strongly convex case violates the lower bound $O\lb\frac{\kappa}{\eps}\rb$ by \cite{woodworth16} because our proof relies on a wrong statement that $\E_{i_{k}}\nuj_{k} = \nabla f(\xj_{k})$. We correct the mistake in this version by using a more delicate derivation. The results for the non-strongly convex case still hold while the results for the strongly convex case is worsen to $\td{O}\lb\frac{\kappa}{\eps}\rb$.}}}\label{sec:ana_SCSG}

Due to the technical complications, we discuss the non-strongly convex case and the strongly convex cases separately in the general case. Similar to R-SVRG, we first derive the sub-optimality bound for $\bar{x}_{T}$. 

\begin{theorem}\label{thm:SCSG_bound}
Assume that $\eta L < \frac{1}{13}$. Under the assumption \textbf{A}1,
\[\displaystyle \E (f(\bar{x}_{T}) - f(x^{*}))\le \frac{1}{T}\cdot\frac{4\etaj B \Delta_{f} + \Delta_{x}}{2\etaj B (1 - 13\etaj L)} + \frac{9\etaj\H \cdot I(\Bj < n)}{2(1 - 13\etaj L)}.\]
\end{theorem}

Note that the bound in Theorem \ref{thm:SCSG_bound} can be simplified as $O\lb \frac{\Delta_{f}}{T} + \frac{\Delta_{x}}{T\etaj B} + \etaj \H\rb$ while the bound in Theorem \ref{thm:SVRG_bound} can be simplified as $O\lb \frac{\etaj \Delta_{f}}{T} + \frac{\Delta_{x}}{T\etaj n}\rb$. Despite the more stringent requirement on $\eta$ ($\eta < \frac{1}{13L}$), these two bounds have two qualitative difference: 1) SCSG has an extra term $O(\etaj \H)$, which characterizes the sampling variance of the mini-batch gradients; 2) SCSG loses an $\eta$ in the first term, which is due to the bias of $\nu_{k}^{(j)}$. In fact, recall the definition of $\E_{i_{k}}$ at the beginning of Section \ref{sec:notation}, a simple calculation shows that 
\[\E_{i_{k}}\nuj_{k} = \nabla f(\xj_{k}) + \lb \nabla f_{\Ij}(\xj_{0}) - \nabla f(\xj_{0})\rb\]
which does not equal $\nabla f(\xj_{k})$ in general. Most novelty of our analysis lies in dealing with the extra bias. Fortunately, we found that the extra terms do not worsen the complexity by scaling $\etaj$ as $\frac{1}{B}$. 

\begin{corollary}\label{cor:SCSG_complexity}
Assume \textbf{A}1 holds. Set 
\[B = \left\lceil \frac{\gamma \H}{L\eps}\wedge n\right\rceil  \quad \mbox{ and }\quad \etaj = \frac{\theta}{L B},\]
Assume that $\displaystyle\frac{\theta}{1 - 13\theta / B}\cdot \frac{9}{2\gamma} < 1$, then with the output $\bar{x}_{T}$,
\begin{equation}\label{eq:SCSG_f}
\E \comp(\eps) = O\lb \frac{\H \Delta_{x}}{\eps^{2}}\wedge \frac{nL\Delta_{x}}{\eps}\rb.
\end{equation}
\end{corollary}

Corollary \ref{cor:SCSG_complexity} shows that SCSG is never worse than SGD and SVRG (with constant stepsize scaled as $\frac{1}{L}$). Compared with SGD whose complexity is $O\lb\frac{\H^{*}\Delta_{x}}{\eps^{2}}\rb$ \cite{mini_batch_SGD} where 
\[\H^{*} = \sup_{x}\frac{1}{n}\sum_{i=1}^{n}\|\nabla f_{i}(x) - \nabla f(x)\|^{2},\]
SCSG has a factor $\H$ which is strictly smaller than $\H^{*}$. It will be shown in Section \ref{sec:Hf}, $\H$ can be much smaller than $\H^{*}$ even in the case where $\H^{*} = \infty$. 

Next we consider the strongly convex case. Similarly, we start from deriving a bound for the sub-optimality of the output $\td{x}_{T}$.
\begin{theorem}\label{thm:SCSG_bound_sc}
Assume that 
\[\phi \triangleq 2 - 8\etaj L - (1 + 13\etaj L )(1 + \mu\etaj B) > 0\]
Under the assumption \textbf{A}1 and assumption \textbf{A}2 with $\mu > 0$, the last iterate $\td{x}_{T}$ satisfies that 
\begin{equation}\label{eq:SCSG_sc_F}
\E (f(\td{x}_{j}) - f(x^{*}))\le \frac{2\Delta_{f}}{\mu\etaj B(1 + \mu\etaj B)^{T}} + \frac{9\H \cdot I(B < n)}{2\mu B}.
\end{equation}
and 
\begin{equation}\label{eq:SCSG_sc_X}
\E \|\td{x}_{j} - x^{*}\|^{2}\le \frac{4}{\mu}\frac{\Delta_{f}}{(1 + \mu\etaj B)^{T}} + \frac{9\etaj \H \cdot I(B < n)}{\mu},
\end{equation}

\end{theorem}

Unlike R-SVRG which guarantees the convergence of $\E f(\td{x}_{T}) - f(x^{*})$ and $\E \|\td{x}_{T} - x^{*}\|^{2}$ simultaneously, SCSG needs to use different batch sizes for the two purposes since the second term in \eqref{eq:SCSG_sc_F} and that in \eqref{eq:SCSG_sc_X} are in different scales. The following two corollaries show the setups for these two purposes.

\begin{corollary}\label{cor:SCSG_complexity_comp_last}
Assume \textbf{A}1 and \textbf{A}2 hold. Set 
\[B = \left\lceil \frac{\gamma (\H\vee L\eps)}{\mu\eps}\wedge n\right\rceil  \quad \mbox{ and }\quad \etaj = \frac{\theta}{\mu (B\vee \gamma\kappa)}.\]
Assume that $\theta \le \frac{\gamma}{22}, \gamma > \frac{9}{2}$, then with the output $\bar{x}_{T}$,
\begin{equation}\label{eq:SCSG_sc_f}
\E \comp(\eps) = O\lb \lb\frac{\H}{\mu \eps}\wedge n +  \kappa\rb\log \lb\frac{\Delta_{f}}{\eps}\cdot\frac{n\vee \kappa}{n}\rb\rb.
\end{equation}
\end{corollary}

\begin{corollary}\label{cor:SCSG_complexity_comp_last_iter}
Under the same settings of Corollary \ref{cor:SCSG_complexity_comp_last} except that setting
\[B = \left\lceil \frac{\gamma (\H\vee L\eps)}{\mu^{2}\eps}\wedge n\right\rceil\]
it holds that 
\begin{equation}\label{eq:SCSG_sc_x}
\E \comp_{x}(\eps) = O\lb \lb\frac{\H}{\mu^{2} \eps}\wedge n +  \kappa\rb\log \lb\frac{\Delta_{f}}{\mu\eps}\rb\rb.
\end{equation}
\end{corollary}

For large $\eps$, ignoring the log-factors, the complexity results \eqref{eq:SCSG_sc_f} and \eqref{eq:SCSG_sc_x} can be simplified as $\td{O}\lb\frac{\H}{\mu \eps}\rb$ and $\td{O}\lb \frac{\H}{\mu^{2}\eps}\rb$. By contrast, the complexity results of SGD are $O\lb\frac{\H^{*}}{\mu \eps}\rb$ and $O\lb\frac{\H^{*}}{\mu^{2}\eps}\rb$, respectively \cite{Rakhlin12}. Thus, SCSG is not worse than SGD up to a log-factor and could significantly outperform SGD when $\H <\!\!<\H^{*}$ in terms of the theoretical complexity. For small $\eps$, SCSG is equivalent to SVRG provided $\kappa = O(n)$, which is usually the case in practice. 

\section{More Details on $\H(f)$}\label{sec:Hf}
The problem \eqref{eq:obj_generic} we considered in this paper is a finite-sum optimization. It is popular to view it under the framework of stochastic approximation (SA) \cite{robbins51} by rewriting $f(x)$ as $\E_{\mathcal{J}}f_{\mathcal{J}}(x) $ where $\mathcal{J}$ is a uniform index on $[n]$ and setting the first-order oracle as drawing $\nabla f_{\mathcal{J}}(x)$ in every step. Then it is necessary to assume that $\H^{*}(f) = \E_{\mathcal{J}} \|\nabla f_{\mathcal{J}}(x) - \nabla f(x)\|^{2}$, as the variance of the oracle output, is uniformly bounded over the domain. However, one should expect that  the finite-sum optimization is strictly easier than the general SA due to the special structure. This paper provides an affirmative answer by introducing a new measure $\H(f)$ to characterize the difficulty of a generic finite-sum optimization problem and developing the SCSG algorithm to adapt to this measure. 

Before delving into the details of $\H(f)$, we briefly review the existing difficulty measures for problem \eqref{eq:obj_generic}.  To the best of our knowledge, the existing measures fall into four categories: \emph{initialization, curvature, gradient regularity} and \emph{heterogeneity}; see Table \ref{tab:difficulty} for corresponding measures. The first three categories of measures are used in almost all types of problems while the heterogeneity measures are specific to the form \eqref{eq:obj_generic}. To illustrate the importance of heterogeneity, consider a  toy example  where $f_{i}(x) = (x - b_{i})^{2}$ with $b_{1}, \ldots, b_{n}\in \R$. Now consider two classes of problem where the first class assumes the prior knowledge that all $b_{i}$'s are equal and the second class assumes that $b_{i}$'s are all free parameters. A simple calculation shows that $\Delta_{x}, \Delta_{f}, \kappa, G^{2}$ are all equal for both classes of problems. However, it is clear that the second class of problems are much easier using stochastic gradient methods since each single function has an exactly the same behavior as the global function. In fact, $\mathcal{G}$ and $\H^{*}$ are zero for the second class of problems while are non-zero for the first class. This suggests that heterogeneity between single functions and the global function increases the difficulty of problem \eqref{eq:obj_generic}. 

\begin{table}
\caption{Existing difficulty measures of problem \eqref{eq:obj_generic} in four categories}\label{tab:difficulty}
\begin{center}
  \begin{tabular}{l|l}
    \toprule
    Categories & Measures\\
    \midrule
    Initialization & $\Delta_{x} = \|x_{0} - x^{*}\|^{2}, \Delta_{f} = f(x_{0}) - f(x^{*})$\\
    Curvature & $\kappa = L / \mu$ (when $\mu > 0$)\\
    Gradient Regularity & $G^{2} = \max_{i}\sup_{x}\|\nabla f_{i}(x)\|^{2}$\\
    Heterogeneity & $\mathcal{G} = \sup_{x}\max_{i}\|\nabla f_{i}(x)\|/\|\nabla f(x)\|$, \\
    & $\H^{*}  = \sup_{x}\frac{1}{n}\sum_{i=1}^{n}\|\nabla f_{i}(x) - \nabla f(x)\|^{2}$\\
    \bottomrule
  \end{tabular}
\end{center}
\end{table}

The first attempt to describe the heterogeneity is through an unrealistic condition, called \emph{strong growth condition}(\cite{schmidt12}), which requires 
\begin{equation}\label{eq:strong_growth}
\max_{i}\|\nabla f_{i}(x)\|\le \mathcal{G}\cdot\|\nabla f(x)\|.
\end{equation}
Under \eqref{eq:strong_growth}, \citet{schmidt12} proves that the stochastic gradient methods have the same convergence rate as the full gradient methods. However, \eqref{eq:strong_growth} is unrealistic since it implies for any minimizer $x^{*}$ of $f$, $x^{*}$ is the stationary point of all individual loss functions. 

Later \citet{mini_batch_SGD} proposed a more realistic measure 
\begin{equation}\label{eq:Hstar}
\H^{*}= \sup_{x}\frac{1}{n}\sum_{i=1}^{n}\|\nabla f_{i}(x) - \nabla f(x)\|^{2}
\end{equation}
and proved that (mini-batch) SGD is adaptive to $\H^{*}$. The condition $\H^{*} < \infty$ is always weaker than assuming $\|\nabla f_{i}\|$ are uniformly bounded in that $\H^{*}\le G^{2}$. However, in many applications where the domain of $x$ is non-compact, $\H^{*} = \infty$. This can be observed even in our toy example when the domain of $x$ is $\R$. One might argue that a projection step may be involved to ensure the boundedness of the iterates. However this argument is quite weak in that 1) the right size of the set that is projected onto is unknown; 2) the projection step is rarely implemented in practice. Therefore, $\H^{*}$ is still not a desirable measure.

By contrast, our proposed measure $\H$ is well-behaved in most applications without awkward assumptions such as the bounded domain. Recall that 
\begin{equation}
  \label{eq:H}
  \H = \inf_{x^{*}\in \argmin f(x)}\frac{1}{n}\sum_{i=1}^{n}\|\nabla f_{i}(x^{*})\|^{2}.
\end{equation}
It can be viewed as a version of $\H^{*}$ which replaces the supremum by the value at a single point, when the optimum of $f(x)$ is unique. As a consequence, $\H\le \H^{*}$. In addition, when the strong growth condition \eqref{eq:strong_growth} holds, $\|\nabla f_{i}(x^{*})\| = 0$ for all $i$ and hence $\H^{*} = 0$. These simple facts show that $\H$ is strictly better than $\mathcal{G}$ and $\H^{*}$ as a measure of difficulty. We will show in the next two subsections that $\H$ can be controlled and estimated in almost all problems and is well-behaved in a wide range of applications.

\subsection{Bounding $\H(f)$ in General Cases}
Although being unrealistic, it is often assumed that $\|\nabla f_{i}(x)\|$ is uniformly bounded over the domain. This implies the boundedness of $\H$ directly and hence provides an example where the problem \eqref{eq:obj_generic} is ``easy''.
\begin{proposition}\label{prop:bound_H}
Let $G^{2}$ and $\H^{*}$ be defined in Table \ref{tab:difficulty}, then 
\[\H\le \min\left\{G^{2}, \H^{*}\right\}.\]
\end{proposition}

Surprisingly, $\H$ can be bounded even without any assumption other than \textbf{A}1 by using an arbitrary reference point.

\begin{proposition}\label{prop:general_H}
Under Assumption \textbf{A}1, for any $x\in \R^{d}$
\begin{equation}\label{eq:general_f} 
\H\le \frac{2}{n}\sum_{i=1}^{n} \lb\|\nabla f_{i}(x)\|^{2} + 2 L(f_{i}(x) - f_{i}(x^{*}))\rb
\end{equation}
\end{proposition}

A natural choice is to set the reference point $x = \td{x}_{0}$.  Under the streaming settings where $f_{i}$ are i.i.d. functions with $\E \|\nabla f_{i}(\td{x}_{0})\|^{2} < \infty$ and $\E |f_{i}(x)| < \infty$ for $x\in \{\td{x}_{0}, x^{*}\}$, the strong law of large number implies that 
\begin{align}
&\frac{1}{n}\sum_{i=1}^{n}\lb\|\nabla f_{i}(\td{x}_{0})\|^{2} + 2f_{i}(x) - 2f_{i}(x^{*})\rb \stackrel{a.s.}{\rightarrow} \E \lb \|\nabla f_{1}(\td{x}_{0})\|^{2} + 2f_{1}(x) - 2_{1}(x^{*})\rb < \infty\nonumber\\
& \qquad\qquad  \Longrightarrow \H = O_{p}(1)\label{eq:iid_f}
\end{align}
This entails that problem \eqref{eq:obj_generic} with i.i.d. individual functions is ``easy''. This is heuristically reasonable since the i.i.d. assumption, plus the moment conditions, forces the data to be highly homogeneous. 

In fact, \eqref{eq:iid_f} can be proved under much broader settings. For example, when solving a linear equation $Ax = b$, $f_{i}(x)$ can be set as $\frac{1}{2}(a_{i}^{T}x - b_{i})^{2}$ and no randomness is involved. If we set $x = 0$, then \eqref{eq:general_f} implies that 
\[\H \le \frac{2}{n}\sum_{i=1}^{n}(2b_{i}^{2} - f_{i}(x^{*}))\le \frac{4\|b\|^{2}}{n}.\]
Then $\H = O(1)$ provided $\|b\|^{2} = O(n)$. 

Another type of problems with $\H = O(1)$ involves pairwise comparisons, i.e.
\[f(x) = \sum_{j,k=1}^{m}f_{jk}(Z_{j}, Z_{k}; x) \]
where $Z_{1}, \ldots, Z_{m}$ are independent samples. For example, in preference elicitation or sporting competitions where the data is collected as pairwise-comparisons, one can fit a Bradley-Terry model to obtain the underlying ``score'' that represents the quality of each unit. The objective function of the Bradley-Terry model is $\sum_{j,k}[W_{j,k}\beta_{j} - W_{j,k}\log (e^{\beta_{j}} + e^{\beta_{k}})]$ where $W_{j, k}$ is the number of times that the unit $j$ beats the unit $k$ (\cite{bradley52}, \cite{hunter04}). Other examples that involve a similar structure are metric learning (\cite{xing02}, \cite{weinberger06}) and convex relaxation of graph cuts (\cite{graphcut}). In these cases, we can also bound $\H$ under mild conditions.
\begin{proposition}\label{prop:pairwise}
Let $U_{jk} = \|\nabla f_{jk}(Z_{j}, Z_{k}; \td{x}_{0})\|^{2} + 2L(f_{jk}(Z_{j}, Z_{k}; \td{x}_{0}) - f_{jk}(Z_{j}, Z_{k}; x^{*}))$. Then 
\[\max_{j,k}\E U_{jk}^{2} = O(1)\Longrightarrow \H = O_{p}(1).\]
\end{proposition}

Finally, it is worth mentioning that \eqref{eq:iid_f} cannot be established for $\H^{*}$ unless the domain is compact and more regularity conditions, than the existence of second moment, are imposed to ensure that a certain version of uniform law of large number can be applied. 

\subsection{Estimating $\H(f)$ in Generalized Linear Models}
Optimzation problems in machine learning are often generalized linear models where $f_{i}(x) = \rho(y_{i}, a_{i}^{T}x)$, with $a_{i}$ being the covariates and $y_{i}$ being the responses, for some convex loss function $\rho$. Let $\rho_{2}(z, w) = \frac{\partial}{\partial w}\rho(z, w)$. Then by definition
\[\H = \frac{1}{n}\sum_{i=1}^{n}\rho_{2}(y_{i}, a_{i}^{T}x^{*})^{2}\cdot \|a_{i}\|^{2}.\]
If $\rho_{2}(y_{i}, a_{i}^{T}x)$ is uniformly bounded with $\rho_{2}(y_{i}, a_{i}^{T}x^{*})^{2} \le M_{1}$, then 
\[\H \le M_{1}\cdot \frac{1}{n}\sum_{i=1}^{n}\|a_{i}\|^{2}.\]
We will show in appendix \ref{app:section4} that $M_{1} = 2$ for multi-class logistic regression, regardless of the number of classes. The same bound can also be derived for Huber regression \cite{HuberBook81}, Probit model (\cite{GLM}), etc.. When the domain is unbounded, the (penalized) least square regression has an unbounded $\rho_{2}(z, w)$. However notice that $x^{*} = (A^{T}A)^{-1}A^{T}y$ where $A = (a_{1}, \ldots, a_{n})^{T}$, one can easily show that
\[\H \le \max_{i}\sup\|a_{i}\|^{2}\cdot \frac{\|y\|^{2}}{n}.\]
 
\subsection{$\H(f)$ in Pathological Cases}
The last two subsections exhibit various examples where $\H$ is well controlled. Indeed, there exist pathological cases where $\H$ is large. For instance, let $n$ be an even number and $f_{i}(x) = \frac{1}{2}\|x - \alpha^{(i)}\|^{2}$, where $\alpha^{(i)}\in \R^{n}$ with $\alpha^{(i)}_{i + 1} = \alpha^{(i)}_{i + 2} = \ldots = \alpha^{(i)}_{i + n/2} = 1$ ($\alpha^{(i)}_{n + k} = \alpha^{(i)}_{k}$) and all other elements equal to $0$.   \footnote{We thank Chi Jin for providing the example.} In this case, $x^{*} = \frac{1}{2}\textbf{1}$ and by symmetry $\H = \|\nabla f_{1}(x^{*})\|^{2} = \frac{n}{4}$. Another example is a quadratic function with $f_{i}(x) = \frac{1}{2}(x - i)^{2}$, in which case $x^{*} = \frac{(n + 1)}{4}$ and hence $\H = \Omega(n^{2})$. 

The first example is due to the high dimension. When the dimension is comparable to $n$, even the i.i.d. assumption cannot guarantee a good behavior of $\H$, without further conditions, since the law of large number fails. The second example is due to the severe heterogeneity of components. In fact the $i$-th component reaches its minimum at $x = i$ while the global function reaches its minimum at $x = \frac{n + 1}{4}$ and thus most components behaves completely different from the global function. 

Nevertheless, it is worth emphasizing that SGD also faces with the same issue in these two cases. More importantly, SCSG does not suffer from these undesirable properties since it will choose $B = n$ automatically; See Corollary \ref{cor:SCSG_complexity} to Corollary \ref{cor:SCSG_complexity_comp_last_iter}.

\section{Experiments}\label{sec:experiment}
In this section, we illustrate the performance of SCSG by implementing it
for multi-class logistic regression on the MNIST dataset~\footnote{http://yann.lecun.com/exdb/mnist/.} We normalize 
the data into the range $[0, 1]$ by dividing each entry by $256$.  
No regularization term is added and so the function to be minimized is 
\[f(x) = \frac{1}{n}\sum_{i=1}^{n}\lb \log\lb 1 + \sum_{k=1}^{K - 1}e^{a_{i}^{T}x_{k}}\rb - \sum_{k=1}^{K - 1}I(y_{i} = k)a_{i}^{T}x_{k}\rb,\]
where $n = 60000, K = 10$, $y_{i}\in \{0, 1, \ldots, 9\}$, $a_{i}\in\R^{785}$ including $28\times 28 = 784$ pixels plus an intercept term $1$ and $x = (x_{1}, \ldots, x_{9})\in \R^{785 \times 9} = \R^{7065}$. Direct computation shows that $ \H  = 174.25$ while $\H^{*} = 585.64$. 

The performance is measured by $\log_{10}\|\nabla f(x)\|^{2}$ versus the number of passes of data \footnote{Although beging ideal to report $f(x) - f(x^{*})$, it is not feasible in that $f^{*}$ is unknown. }. For each algorithm mentioned later, we selects the best-tuned stepsize and then implement the algorithm for 10 times and report the average to avoid the random effect. 

Here we compare SCSG with mini-batch SGD, with the batch size $B$, and SVRG. Moreover, we consider three variants of SCSG:
\begin{enumerate}[(1)]
\item (SCSGFixN) set $\Nj \equiv B$, instead of generated from a geometric distribution;
\item (SCSGNew) randomly pick $i_{k}\in \I_{j}$, instead of from the whole dataset $[n]$;
\item (SCSGNewFixN) set $\Nj \equiv B$ and randomly pick $i_{k}\in \I_{j}$.
\end{enumerate}
The first variant is to check whether geometric random variable is essential in practice; the second variant is to check whether running SGD from the whole dataset is necessary; and the third variant is the combination.

For all the variants of SCSG and SGD, we consider three batch sizes $B\in \{0.01n, 0.05n, 0.25n\}$. The results are plotted in Figure \ref{fig:mnist}, from which we make the following observations: 
\begin{enumerate}[1)]
\item SCSG is able to reach an accurate solution very fast since all versions of SCSG are more efficient than SGD and SVRG in the first 5 passes. This confirms our theory;
\item SCSG with fixed $\Nj$ is slightly more effective than the original SCSG. Thus the geometric random variable might not be essential in practice;
\item  It makes no difference whether sampling from the whole dataset or sampling from the mini-batch when running the SGD steps in SCSG. 
\end{enumerate}

\begin{figure}[htp]
  \centering
  \includegraphics[width = 0.32\textwidth]{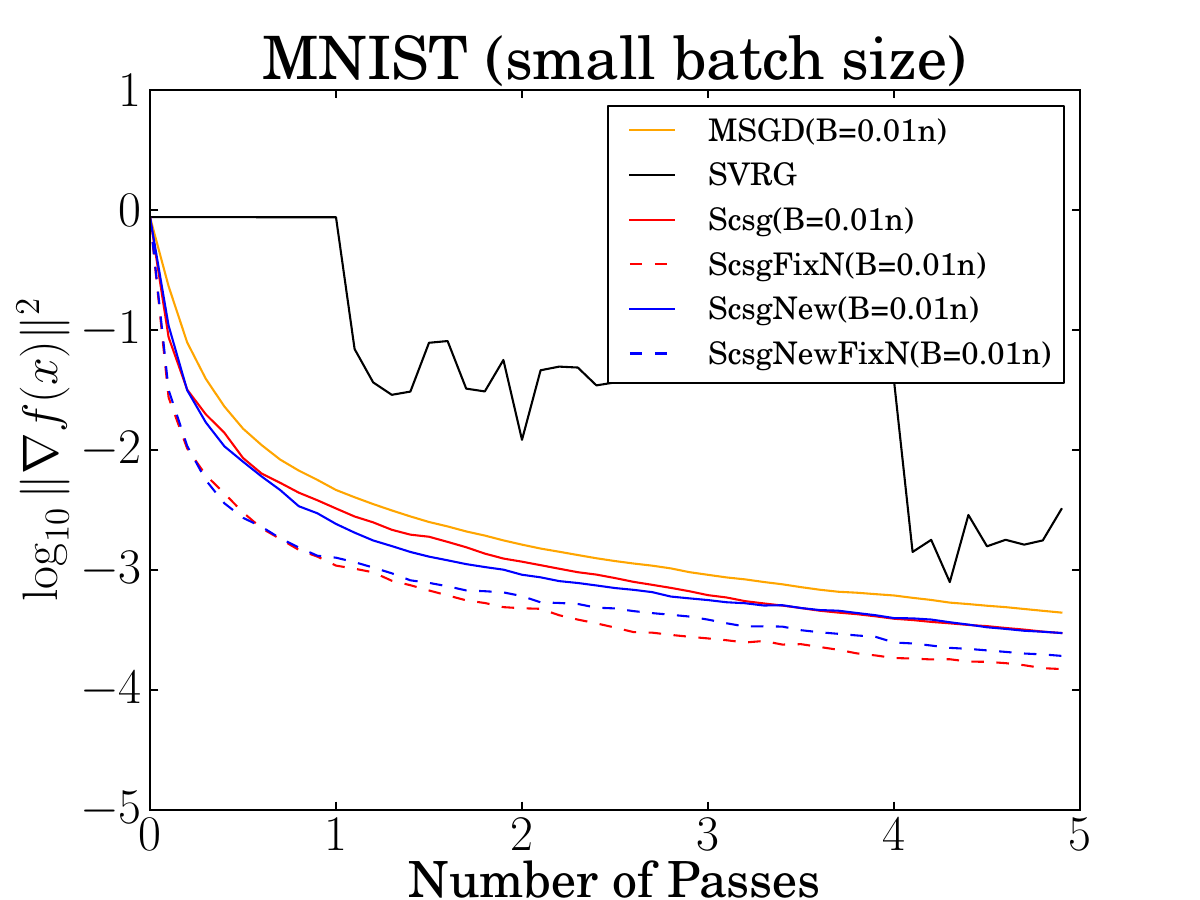}
  \includegraphics[width = 0.32\textwidth]{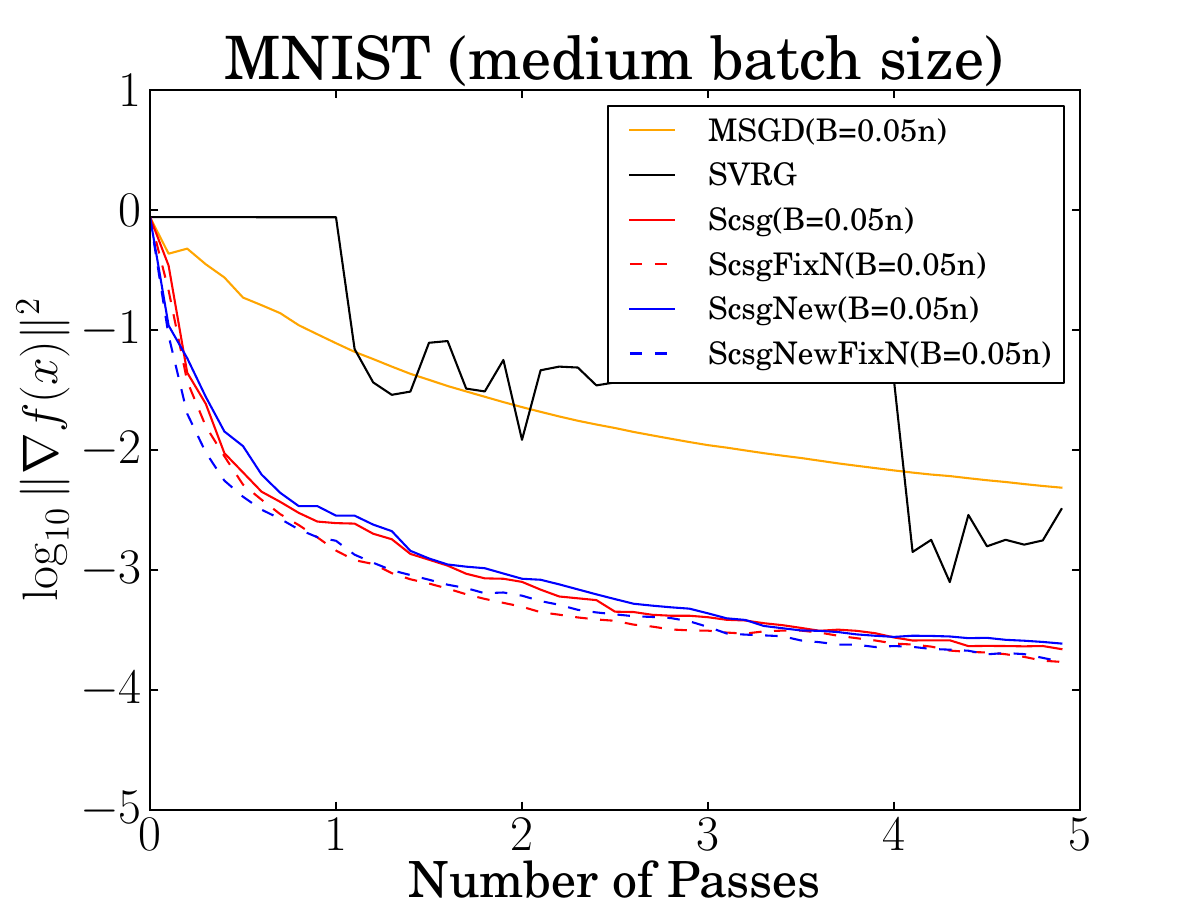}
  \includegraphics[width = 0.32\textwidth]{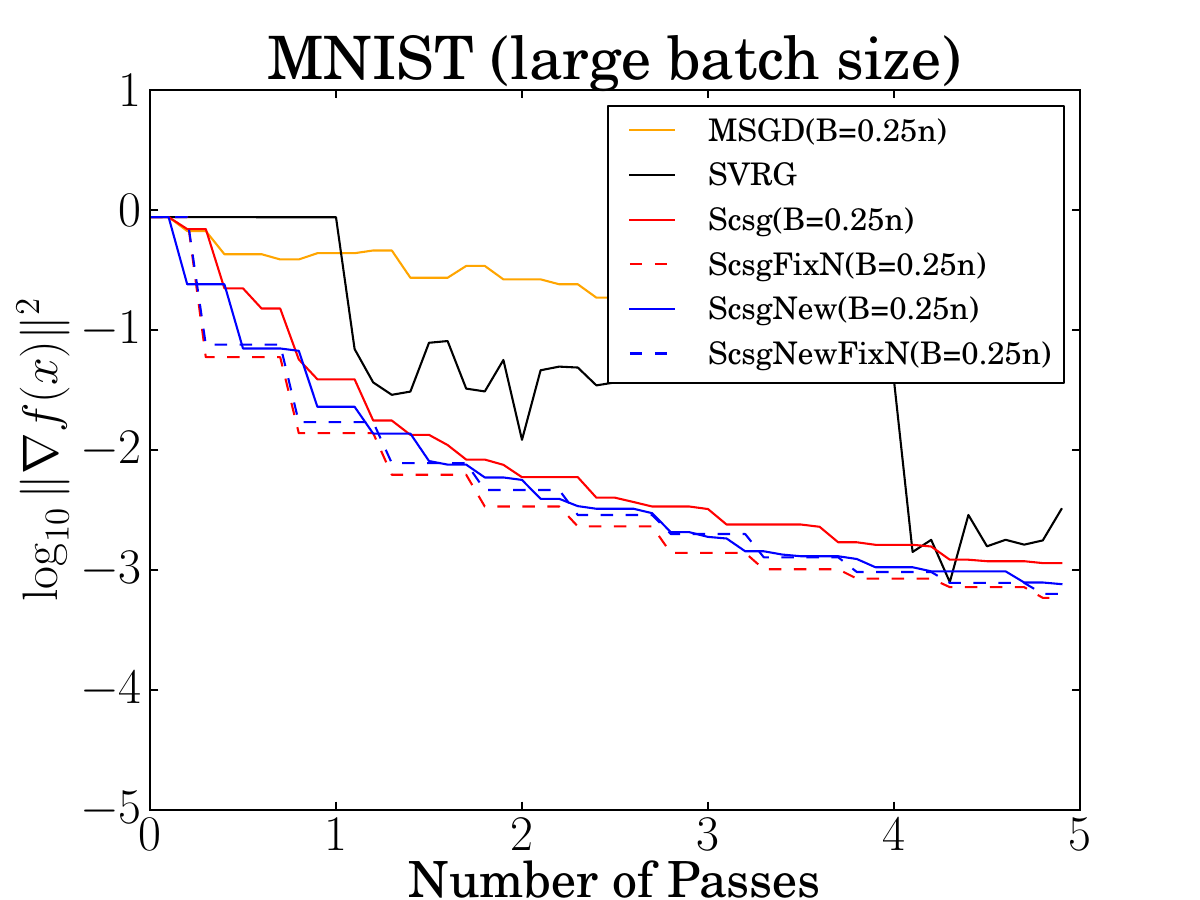}
  \caption{Performance plots on SCSG and other algorithms. Each column represents a (initial )batch size ($0.01n, 0.05n$ and $0.25n$)}\label{fig:mnist}
\end{figure}

Based on our observations, we recommend implementing SCSGNewFixN as the default since 1) the fixed number of SGD steps stablizes the procedure; 2) sampling from the mini-batch reduces the communication cost incurred by accessing data from the whole dataset. 
 
\section{Discussion}\label{sec:discussion}
We propose SCSG as a member of the SVRG family of algorithms, proving its superior performance in terms of both computation and communication cost. Both complexities are independent of sample size when the required accuracy is low, for various functions which are widely optimized in practice. The real data example also validates our theory.

We plan to explore several variants of SCSG in future work. For example, a non-uniform sampling scheme can be applied to SGD steps to leverage the Lipschitz constants $L_{i}$ as in SVRG. More interestingly, we can consider a better sampling scheme for $\I_{j}$ by putting more weight on influential observations. The proximal settings are also straightforward extensions of our current work.

As a final comment, we found that the previous complexity analysis focuses on the high-accuracy computation for which the dependence on the sample size $n$ and condition number $\kappa$ is of major concern. The low-accuracy regime is unfortunately under-studied theoretically even though it is commonly encountered in practice. We advocate taking all three parameters, namely $n$, $\kappa$ and $\eps$, into consideration and distinguishing the analyses for high-accuracy computation and low-accuracy computation as standard practice in the literature.

\section{Acknowledgment}
We thank the Chi Jin, Nathan Srebro and anonymous reviewers for their helpful comments, which greatly improved this work. 


\newpage
\bibliography{SCSG}
\bibliographystyle{plain}

\appendix
\section{Lemmas}\label{app:lemmas}
\begin{lemma}\label{lem:cocoercive}
Let $g$ be a convex function that satisfies the assumption $\textbf{A}1$,
\begin{align*}
\|\nabla g(x) - \nabla g(y)\|^{2}&\le 2L(\breg{g}{x}{y})
\end{align*}
\end{lemma}
\begin{proof}
This is the standard Co-coercivity argument; See e.g. \cite{cocoercive}, Theorem 2.1.5 of \cite{nesterov04}.
\end{proof}

\begin{lemma}\label{lem:geom}
Let $N\sim \mathrm{Geom}(\gamma)$ for some $\gamma > 0$. Then for any sequence $\{D_{n}\}$ with $\E |D_{N}| < \infty$,
\[\E (D_{N} - D_{N + 1}) = \lb\frac{1}{\gamma} - 1\rb\lb D_{0} - \E D_{N}\rb.\]
\end{lemma}
\begin{proof}
By definition, 
  \begin{align*}
   &\E (D_{N} - D_{N + 1})  = \sum_{n\ge 0}(D_{n} - D_{n+1})\cdot \gamma^{n}(1 - \gamma)\\ 
= &(1 - \gamma)\lb D_{0} - \sum_{n\ge 1}D_{n}(\gamma^{n-1} - \gamma^{n})\rb = (1 - \gamma)\lb \frac{1}{\gamma}D_{0} - \sum_{n\ge 0}D_{n}(\gamma^{n-1} - \gamma^{n})\rb\\ 
= & (1 - \gamma)\lb \frac{1}{\gamma}D_{0} - \frac{1}{\gamma}\sum_{n\ge 0}D_{n}\gamma^{n}(1 - \gamma)\rb = \lb\frac{1}{\gamma} - 1\rb (D_{0} - \E D_{N}),
  \end{align*}
where the last equality follows the condition that $\E |D_{N}| < \infty$.
\end{proof}

\begin{lemma}\label{lem:quad}
  Let $a, b> 0$ and $c\in \R$. Then for any $x \ge 0$,
\[ax^{2} \le bx + c\Longrightarrow ax^{2} \le \frac{b^{2}}{a} + 2c.\]
\end{lemma}
\begin{proof}
  An elementary computation shows that 
\[ax^{2} \le bx + c\Longrightarrow x\le  \frac{b}{2a} + \sqrt{\frac{c}{a} + \frac{b^{2}}{4a^{2}}}.\]
Using the fact that $(z + w)^{2}\le 2z^{2} + 2w^{2}$, we have
\[x^{2}\le 2\lb\frac{b}{2a}\rb^{2} + 2\lb\frac{c}{a} + \frac{b^{2}}{4a^{2}}\rb = \frac{b^{2}}{a^{2}} + \frac{2c}{a}.\Longrightarrow ax^{2} \le \frac{b^{2}}{a} + 2c.\]
\end{proof}

\section{One-Epoch Analysis}\label{app:one_epoch}
First we prove a lemma that generalizes Lemma \ref{lem:sg_optim_var}.
\begin{lemma}\label{lem:var_sampling}
Let $z_{1}, \ldots, z_{M}\in \R^{d}$ be an arbitrary population of $M$ vectors with 
\[\sum_{j=1}^{M}z_{j} = 0.\]
Further let $\mathcal{J}$ be a uniform random subset of $[M]$ with size $m$. Then 
\[\E \left\|\frac{1}{m}\sum_{j\in \mathcal{J}}z_{j}\right\|^{2} = \frac{M - m}{(M - 1)m}\cdot \frac{1}{M}\sum_{j=1}^{M}\|z_{j}\|^{2}\le \frac{I(m < M)}{m}\cdot \frac{1}{M}\sum_{j=1}^{M}\|z_{j}\|^{2}.\]
\end{lemma}
\begin{proof}
Let $W_{j} = I(j\in \mathcal{J})$, then it is easy to see that
\begin{equation}
  \label{eq:cov_wj}
  \E W_{j}^{2} = \E W_{j} = \frac{m}{M}, \quad \E W_{j}W_{j'} = \frac{m(m - 1)}{M(M - 1)}.
\end{equation}
Then the sampling mean can be reformulated as
\[\frac{1}{m}\sum_{j\in \mathcal{J}}z_{j} = \frac{1}{m}\sum_{i=1}^{n}W_{j}z_{j}.\]
This implies that 
\begin{align*}
\E \bigg\|\frac{1}{m}\sum_{j\in \mathcal{J}}z_{j}\bigg\|^{2} &= \frac{1}{m^{2}}\lb\sum_{j=1}^{M}\E W_{j}^{2}\|z_{j}\|^{2} + \sum_{j\not= j'}\E W_{j}W_{j'} \la z_{j}, z_{j'}\ra\rb\\
& = \frac{1}{m^{2}}\lb\frac{m}{M}\sum_{j=1}^{M}\|z_{j}\|^{2} + \frac{m(m - 1)}{M(M - 1)}\sum_{j\not= j'} \la z_{j}, z_{j'}\ra\rb\\
& = \frac{1}{m^{2}}\lb\lb\frac{m}{M} - \frac{m(m - 1)}{M(M - 1)}\rb\sum_{j=1}^{M}\|z_{j}\|^{2} + \frac{m(m - 1)}{M(M - 1)}\left\|\sum_{j=1}^{M}z_{j}\right\|^{2}\rb\\
& = \frac{1}{m^{2}}\lb\frac{m}{M} - \frac{m(m - 1)}{M(M - 1)}\rb\sum_{j=1}^{M}\|z_{j}\|^{2}\\
& = \frac{M - m}{(M - 1)m}\cdot\frac{1}{M}\sum_{j=1}^{M}\|z_{j}\|^{2}.
\end{align*}
\end{proof}

\begin{proof}[\textbf{Lemma \ref{lem:sg_optim_var}}]
Let $z_{i}=  \nabla f_{i}(x^{*})$. Then 
\[\sum_{i=1}^{n}z_{i} = n\nabla f(x^{*}) = 0.\]
By Lemma \ref{lem:var_sampling} we prove the result.
\end{proof}

As all other algorithms, we start from deriving a bound for the stochastic gradients $\nuj_{k}$. For convenience, we define $\ej$ as the bias of $\nuj_{k}$, i.e.
\begin{equation}
  \label{eq:ej}
  \ej = \nabla f_{\I_{j}}(\xj_{0}) - \nabla f(\xj_{0}).
\end{equation}
By definition,
\[\E_{i_{k}}\nuj_{k} = \nabla f(\xj_{k}) + \ej.\]

\begin{lemma}\label{lem:nuknorm}
Under the assumption \textbf{A}1 and \textbf{A}2 with $\mu$ possibly equal to $0$,
\begin{align} 
  \E_{i_{k}}\|\nuj_{k}\|^{2} & \le -4L\lb f(\xj_{k}) - f(x^{*})\rb + 4L\lb f(\xj_{0}) - f(x^{*})\rb + 4L\la\nabla f(\xj_{k}), \xj_{k} - x^{*}\ra\nonumber\\
& \quad  + \|\nabla f(\xj_{k}) + \ej\|^{2}.\label{eq:nuknorm_tmp}
\end{align}
\end{lemma}
\begin{proof}
  Using the fact that $\E \|Z\|^{2} = \E \|Z - \E Z\|^{2} + \|\E Z\|^{2}$ (for any random variable $Z$), we have
  \begin{align*}
&\E_{i_{k}}\|\nuj_{k}\|^{2}  = \E_{i_{k}}\|\nuj_{k} - \E_{i_{k}}\nuj_{k}\|^{2} + \|\E_{i_{k}}\nuj_{k}\|^{2}\\
 = & \E_{i_{k}}\|\nabla f_{i_{k}}(\xj_{k}) - \nabla f_{i_{k}}(\xj_{0}) - (\nabla f(\xj_{k}) - \nabla f(\xj_{0}))\|^{2} + \|\nabla f(\xj_{k}) + \ej\|^{2}\\
\le & \E_{i_{k}}\|\nabla f_{i_{k}}(\xj_{k}) - \nabla f_{i_{k}}(\xj_{0})\|^{2} + \|\nabla f(\xj_{k}) + \ej\|^{2}\\
\le & 2\E_{i_{k}}\|\nabla f_{i_{k}}(\xj_{k}) - \nabla f_{i_{k}}(x^{*})\|^{2} + 2\E_{i_{k}}\|\nabla f_{i_{k}}(\xj_{0}) - \nabla f_{i_{k}}(x^{*})\|^{2} + \|\nabla f(\xj_{k}) + \ej\|^{2}.
\end{align*}
By Lemma \ref{lem:cocoercive} with $g = f_{i_{k}}, x = x^{*}, y = \xj_{k}$,
\begin{align*}
\E_{i_{k}}\|\nabla f_{i_{k}}(\xj_{k}) - \nabla f_{i_{k}}(x^{*})\|^{2}&\le 2L\E_{i_{k}}\lb f_{i_{k}}(x^{*}) - f_{i_{k}}(\xj_{k}) + \la\nabla f_{i_{k}}(\xj_{k}), \xj_{k} - x^{*}\ra\rb\\
& = 2L \lb f(x^{*}) - f(\xj_{k}) + \la\nabla f(\xj_{k}), \xj_{k} - x^{*}\ra\rb 
\end{align*}
where the last line uses the fact that $i_{k}$ is independent of $(\xj_{k}, \xj_{0})$. Similarly, by Lemma \ref{lem:cocoercive} with $g = f_{i_{k}}, x = \xj_{0}, y = x^{*}$, 
\begin{align*}
\E_{i_{k}}\|\nabla f_{i_{k}}(\xj_{0}) - \nabla f_{i_{k}}(x^{*})\|^{2}&\le 2L\E_{i_{k}}\lb f_{i_{k}}(\xj_{0}) - f_{i_{k}}(x^{*}) - \la\nabla f_{i_{k}}(x^{*}), \xj_{0} - x^{*}\ra\rb\\
& = 2L\lb f(\xj_{0}) - f(x^{*})\rb.
\end{align*}
where the last line uses the smoothness of $f$. Putting the pieces together, we conclude that
\begin{align*}
  \E_{i_{k}}\|\nuj_{k}\|^{2} & \le -4L\lb f(\xj_{k}) - f(x^{*})\rb + 4L\lb f(\xj_{0}) - f(x^{*})\rb + 4L\la\nabla f(\xj_{k}), \xj_{k} - x^{*}\ra\\
& \quad  + \|\nabla f(\xj_{k}) + \ej\|^{2}.
\end{align*}
\end{proof}

Note that $\ej = 0$ when $B = n$. Thus, in the analysis of R-SVRG, the last term of \eqref{eq:nuknorm_tmp} reduces to $\E \|\nabla f(\xj_{k})\|^{2}$. For general case $B < n$, we will relax the last term of  \eqref{eq:nuknorm_tmp} by using the simple inequality that $\|a + b\|^{2}\le 2\|a\|^{2} + 2\|b\|^{2}$.

\begin{corollary}\label{cor:nuknorm2}
Under the same setting as Lemma \ref{lem:nuknorm}, 
\begin{align*} 
  \E_{i_{k}}\|\nuj_{k}\|^{2} & \le -4L\lb f(\xj_{k}) - f(x^{*})\rb + 4L\lb f(\xj_{0}) - f(x^{*})\rb + 4L\la\nabla f(\xj_{k}), \xj_{k} - x^{*}\ra\\
& \quad  + 2\|\nabla f(\xj_{k})\|^{2} + 2\|\ej\|^{2}.
\end{align*}
\end{corollary}
To apply the property of geometric random variables (Lemma \ref{lem:geom}), we need to justify the condition $\E |D_{N}| < \infty$ for different choices of $\{D_{k}\}$. The proof is distracting and relegated to the end of this subsubsection. 
\begin{lemma}\label{lem:geom_finite}
Assume that $\etaj L \le 1 / 2$. Then for any $j$, 
\[\E \|\td{x}_{j} - x^{*}\|^{2} < \infty, \quad \E \|\nuj_{\Nj}\|_{2}^{2} < \infty, \quad \E |\la \ej, \td{x}_{j} - \td{x}_{j-1}\ra| < \infty.\]
\end{lemma}

The next step is to bound the dual gap \cite{zhu14}. Due to the bias of $\nuj_{k}$, we will have an extra term $2\etaj\Bj\sqrt{\E\|\ej\|^{2}}\cdot\sqrt{\E\|\td{x}_{j} - \td{x}_{j-1}\|^{2}}$ compared to the standard analysis.

\begin{lemma}\label{lem:dual}
Let $u\in \R^{d}$ be any variable that is independent of $\Ij$ and subsequent random indices within the $j$-th epoch, $i_{0}, i_{1}, \ldots$. Then 
\begin{align*}
2\etaj\Bj\E\la \nabla f(\td{x}_{j}), \td{x}_{j} - u \ra &\le\E\|\td{x}_{j-1} - u\|^{2} - \E\|\td{x}_{j} - u\|^{2} + 2\etaj\Bj\sqrt{\E\|\ej\|^{2}}\cdot\sqrt{\E\|\td{x}_{j} - \td{x}_{j-1}\|^{2}}\\
&  + \etaj^{2}\Bj\E \|\nuj_{\Nj}\|^{2}.
\end{align*}  
\end{lemma}
\begin{proof}
By definition,
  \begin{align*}
&\E_{i_{k}}\|\xj_{k+1} - u\|^{2}  = \|\xj_{k} - u\|^{2} - 2\etaj\E_{i_{k}}\la  \nuj_{k}, \xj_{k} - u\ra + \etaj^{2} \E_{i_{k}}\|\nuj_{k}\|^{2}\\
& = \|\xj_{k} - u\|^{2} - 2\etaj\la \E_{i_{k}}\nuj_{k}, \xj_{k} - u\ra + \etaj^{2} \E_{i_{k}}\|\nuj_{k}\|^{2}\\
& = \|\xj_{k} - u\|^{2} - 2\etaj\la \nabla f(\xj_{k}), \xj_{k} - u \ra -2\etaj \la \ej, \xj_{k} - u\ra+ \etaj^{2} \E_{i_{k}}\|\nuj_{k}\|^{2}.
  \end{align*}
Let $\Ej$ denote the expectation with respect to $\I_{j}$ and $i_{0}, i_{1}, \ldots$. Then 
\[\Ej\|\xj_{k+1} - u\|^{2} = \Ej\|\xj_{k} - u\|^{2} - 2\etaj\Ej\la \nabla f(\xj_{k}), \xj_{k} - u \ra -2\etaj \Ej\la \ej, \xj_{k} - u\ra+ \etaj^{2} \E_{j}\|\nuj_{k}\|^{2}.\]
Since $u$ is independent of $\Ij$ and $\E_{\Ij}\ej = 0$, we have
\[\Ej\la \ej, u\ra = \Ej\E_{\I_{j}}\la \ej, u\ra = \Ej \la \E_{\I_{j}}\ej, u\ra = 0.\]
Similarly, since $\xj_{0}$ is also independent of $\Ij$, 
\[\Ej \la\ej, \xj_{0}\ra = 0.\]
Therefore,
\[\Ej\la \ej, \xj_{k} - u\ra  = \Ej\la\ej, \xj_{k} - \xj_{0}\ra.\]
Now let $k = \Nj$ and taking expectation with respect to $\Nj$, by Lemma \ref{lem:geom} and \ref{lem:geom_finite},
\begin{align*}
&2\etaj\E_{\Nj}\Ej\la \nabla f(\xj_{\Nj}), \xj_{\Nj} - u \ra  = \E_{\Nj}\lb\Ej\|\xj_{\Nj} - u\|^{2} - \Ej\|\xj_{\Nj+1} - u\|^{2}\rb \\
& \quad - 2\etaj \E_{\Nj}\Ej\la\ej, \xj_{\Nj} - \xj_{0}\ra + \etaj^{2} \E_{\Nj}\E_{j}\|\nuj_{\Nj}\|^{2}\\
& = \frac{1}{\Bj}\lb \Ej\|\xj_{0} - u\|^{2} - \E_{\Nj}\Ej\|\xj_{\Nj} - u\|^{2}\rb - 2\etaj \E_{\Nj}\Ej\la\ej, \xj_{\Nj} - \xj_{0}\ra + \etaj^{2} \E_{\Nj}\E_{j}\|\nuj_{\Nj}\|^{2}.
\end{align*}
Replacing $\xj_{\Nj}(\xj_{0})$ by $\td{x}_{j}(\td{x}_{j-1})$ by definition and taking further expectation over all past randomness, the above equality can be rewritten as 
\begin{equation}
  \label{eq:dual1}
2\etaj\Bj\E\la \nabla f(\td{x}_{j}), \td{x}_{j} - u \ra =\E\|\td{x}_{j-1} - u\|^{2} - \E\|\td{x}_{j} - u\|^{2} - 2\etaj \Bj\E\la\ej, \td{x}_{j} - \td{x}_{j-1}\ra + \etaj^{2}\Bj \E\|\nuj_{\Nj}\|^{2}.
\end{equation}
The Cauchy-Schwartz inequality implies that 
\[-2\etaj \Bj\E\la \ej, \td{x}_{j} - \td{x}_{j-1}\ra\le 2\etaj\Bj\sqrt{\E\|\ej\|^{2}}\cdot\sqrt{\E\|\td{x}_{j} - \td{x}_{j-1}\|^{2}}.\]
Therefore, we prove the result.
\end{proof}

Applying the Lemma \ref{lem:dual} with $u = x^{*}$, we obtain the key inequality that connects $\td{x}_{j}$, $\td{x}_{j-1}$ and $x^{*}$, which is a standard step in the convergence analysis of other algorithms.
\begin{corollary}\label{cor:dual_zero}
\begin{align*}
&2\etaj\Bj(1 - 2\etaj L)\E\la \nabla f(\td{x}_{j}), \td{x}_{j} - x^{*}\ra \le\E\|\td{x}_{j-1} - x^{*}\|^{2} - \E\|\td{x}_{j} - x^{*}\|^{2} + 4\etaj^{2}L\Bj \E (f(\td{x}_{j-1}) - f(x^{*})) \\
& - 4\etaj^{2}L\Bj \E (f(\td{x}_{j}) - f(x^{*})) + 2\etaj^{2} \Bj \E \|\nabla f(\td{x}_{j})\|^{2} + 2\etaj\Bj\sqrt{\E\|\ej\|^{2}}\cdot\sqrt{\E\|\td{x}_{j} - \td{x}_{j-1}\|^{2}} + 2\etaj^{2}\Bj\E \|\ej\|^{2}.
\end{align*}    
\end{corollary}
\begin{proof}
  Let $u = x^{*}$. By Lemma \ref{lem:nuknorm} and Lemma \ref{lem:dual}, we have
\begin{align}
&2\etaj\Bj\E\la \nabla f(\td{x}_{j}), \td{x}_{j} - x^{*}\ra \le\E\|\td{x}_{j-1} - x^{*}\|^{2} - \E\|\td{x}_{j} - x^{*}\|^{2} + 2\etaj\Bj\sqrt{\E\|\ej\|^{2}}\cdot\sqrt{\E\|\td{x}_{j} - \td{x}_{j-1}\|^{2}}\nonumber\\
& + \etaj^{2}\Bj\bigg( -4L \E (f(\td{x}_{j}) - f(x^{*})) + 4L \E (f(\td{x}_{j-1}) - f(x^{*})) + 4L\E \la\nabla f(\td{x}_{j}), \td{x}_{j} - x^{*}\ra\\
& \qquad \qquad  + 2 \E \|\nabla f(\td{x}_{j})\|^{2} + 2\E \|\ej\|^{2}\bigg).\label{eq:dual_zero_0}
\end{align}  
\end{proof}

To bound the term $\E \|\td{x}_{j} - \td{x}_{j-1}\|^{2}$, we let $u = \td{x}_{j-1}$ in Lemma \ref{lem:dual}. 
\begin{corollary}\label{cor:dual_first}
\begin{align*}
\E \|\td{x}_{j} - \td{x}_{j-1}\|^{2}&\le 4\etaj\Bj\lb 1 + 2\etaj L\rb\E\lb f(\td{x}_{j-1}) - f(x^{*})\rb - 4\etaj \Bj\lb 1 + 2\etaj L\rb \E (f(\td{x}_{j}) - f(x^{*})) \\
& + 8\etaj^{2}L \Bj \E \la\nabla f(\td{x}_{j}), \td{x}_{j} - x^{*}\ra +  4 \etaj^{2}\Bj \E \|\nabla f(\td{x}_{j})\|^{2} + 4\etaj^{2}\Bj^{2}\lb 1 + \frac{1}{\Bj}\rb\E\|\ej\|^{2}.
\end{align*}
\end{corollary}
\begin{proof}
  Let $u = \td{x}_{j-1}$, then it is independent of $\Ij$ and $i_{0}, i_{1}, \ldots$. By Lemma \ref{lem:dual}, we have
\[\E \|\td{x}_{j} - \td{x}_{j-1}\|^{2}\le 2\etaj\Bj\sqrt{\E\|\ej\|^{2}}\cdot\sqrt{\E\|\td{x}_{j} - \td{x}_{j-1}\|^{2}} - 2\etaj\Bj\E \la\nabla f(\td{x}_{j}), \td{x}_{j} - \td{x}_{j-1}\ra + \etaj^{2}\Bj\E \|\nuj_{\Nj}\|^{2}.\]
By Lemma \ref{lem:quad}, we have
\[\E \|\td{x}_{j} - \td{x}_{j-1}\|^{2}\le - 4\etaj\Bj\E \la\nabla f(\td{x}_{j}), \td{x}_{j} - \td{x}_{j-1}\ra + 2\etaj^{2}\Bj\E \|\nuj_{\Nj}\|^{2} + 4\etaj^{2}\Bj^{2}\E\|\ej\|^{2}.\]
Using the convexity of $f$, we have
\[\la\nabla f(\td{x}_{j}), \td{x}_{j} - \td{x}_{j-1}\ra\ge f(\td{x}_{j}) - f(\td{x}_{j-1}) = (f(\td{x}_{j}) - f(x^{*})) - (f(\td{x}_{j-1}) - f(x^{*})).\]
By Corollary \ref{cor:nuknorm2}, 
\begin{align*}
\E \|\nuj_{\Nj}\|^{2}& \le -4L\E (f(\td{x}_{j}) - f(x^{*})) + 4L\E (f(\td{x}_{j-1}) - f(x^{*})) + 4L\E \la \nabla f(\td{x}_{j}), \td{x}_{j} - x^{*}\ra\\
& \quad + 2\E \|\nabla f(\td{x}_{j})\|^{2} + 2\E\|\ej\|^{2}.
\end{align*}
Therefore, we prove the result.
\end{proof}

The last term to bound is $\E \|\ej\|^{2}$. This is a simple consequence of Lemma \ref{lem:var_sampling}.
\begin{lemma}\label{lem:ej}
\[\E \|\ej\|^{2}\le \frac{4L\cdot I(\Bj < n)}{\Bj}\E (f(\td{x}_{j-1}) - f(x^{*})) + \frac{2\H \cdot I(\Bj < n)}{\Bj}.\]  
\end{lemma}

\begin{proof}
Using the fact that $\|z + w\|^{2} \le 2\|z\|^{2} + 2\|w\|^{2}$, we have
\begin{align*}
&  \E \|\ej\|^{2}  \le 2\E \|\ej - \nabla f_{\I_{j}}(x^{*})\| + 2\E \|\nabla f_{\I_{j}}(x^{*})\|^{2}\\
& = 2\E \left\|\frac{1}{B}\sum_{i\in \Ij}(\nabla f_{i}(\td{x}_{j-1}) - \nabla f_{i}(x^{*})) - (\nabla f(\td{x}_{j-1}) - \nabla f(x^{*}))\right\|^{2} + 2\E \left\|\frac{1}{\Bj}\sum_{i\in \Ij}\nabla f_{i}(x^{*})\right\|^{2}.
\end{align*}
Noticing that $\nabla f(x^{*}) = 0$, by Lemma \ref{lem:var_sampling}, 
\begin{align*}
&\E \left\|\frac{1}{B}\sum_{i\in \Ij}(\nabla f_{i}(\td{x}_{j-1}) - \nabla f_{i}(x^{*})) - (\nabla f(\td{x}_{j-1}) - \nabla f(x^{*}))\right\|^{2}\\
& \le \frac{I(\Bj < n)}{\Bj}\cdot \frac{1}{n}\sum_{i=1}^{n}\|\nabla f_{i}(\td{x}_{j-1}) - \nabla f_{i}(x^{*}) - (\nabla f(\td{x}_{j-1})) - \nabla f(x^{*})\|^{2}\\
& = \frac{I(\Bj < n)}{\Bj}\cdot \lb\frac{1}{n}\sum_{i=1}^{n}\|\nabla f_{i}(\td{x}_{j-1}) - \nabla f_{i}(x^{*})\|^{2} - \|\nabla f(\td{x}_{j-1}) - \nabla f(x^{*})\|^{2}\rb\\
& \le \frac{I(\Bj < n)}{\Bj}\cdot \frac{1}{n}\sum_{i=1}^{n}\|\nabla f_{i}(\td{x}_{j-1}) - \nabla f_{i}(x^{*})\|^{2}\\
& \le \frac{I(\Bj < n)}{\Bj}\cdot \frac{2L}{n}\sum_{i=1}^{n}\lb f_{i}(\td{x}_{j-1}) - f_{i}(x^{*})\rb\\
& = \frac{I(\Bj < n)}{\Bj}\cdot 2L(f(\td{x}_{j - 1}) - f(x^{*})).
\end{align*}
On the other hand, by Lemma \ref{lem:var_sampling} again, we obtain that 
\[\E \left\|\frac{1}{\Bj}\sum_{i\in \Ij}\nabla f_{i}(x^{*})\right\|^{2}\le \frac{\H\cdot I(\Bj < n)}{\Bj}.\]
Putting the pieces together we prove the result.
\end{proof}

Putting all the pieces together, we can derive the key inequality on the performance of SCSG within a single epoch.
\begin{theorem}\label{thm:dual_second}
Suppose $\etaj L \le \frac{1}{8}$ and $\Bj \ge 8$. Then under the assumption \textbf{A}1 and \textbf{A}2,
\begin{align*}
&\quad \lb 1 + \mu\etaj\Bj (1 - 8\etaj L)\rb\E\|\td{x}_{j} - x^{*}\|^{2}+ 4\etaj \Bj \E (f(\td{x}_{j}) - f(x^{*}))  \\  
& \le  \E\|\td{x}_{j-1} - x^{*}\|^{2} + 2\etaj\Bj\lb 1+ 4\etaj L + 9\etaj L \cdot  I(\Bj < n)\rb \E (f(\td{x}_{j-1}) - f(x^{*})) + 9\etaj^{2}\Bj \cdot \H\cdot I(\Bj < n).
\end{align*}
  
\end{theorem}
\begin{proof}
By Corollary \ref{cor:dual_zero}, 
\begin{align*}
&2\etaj\Bj(1 - 2\etaj L)\E\la \nabla f(\td{x}_{j}), \td{x}_{j} - x^{*}\ra \le\E\|\td{x}_{j-1} - x^{*}\|^{2} - \E\|\td{x}_{j} - x^{*}\|^{2} + 4\etaj^{2}L\Bj \E (f(\td{x}_{j-1}) - f(x^{*})) \\
& - 4\etaj^{2}L\Bj \E (f(\td{x}_{j}) - f(x^{*})) + 2\etaj^{2}\Bj \E \|\nabla f(\td{x}_{j})\|^{2} + 2\etaj\Bj\sqrt{\E\|\ej\|^{2}}\cdot\sqrt{\E\|\td{x}_{j} - \td{x}_{j-1}\|^{2}} + 2\etaj^{2}\Bj\E \|\ej\|^{2}.
\end{align*}    

Using the fact that $2zw\le \beta^{-1} z^{2} + \beta w^{2}$ for any $\beta > 0$, we have
\begin{equation}
  \label{eq:dual_second_1}
  2\sqrt{\E\|\ej\|^{2}}\cdot\sqrt{\E\|\td{x}_{j} - \td{x}_{j-1}\|^{2}} \le 2\etaj\Bj\E \|\ej\|^{2} + \frac{1}{2\etaj\Bj} \E\|\td{x}_{j} - \td{x}_{j-1}\|^{2}.
\end{equation}
Then by Corollary \ref{cor:dual_first} and Lemma \ref{lem:ej}, we obtain that
\begin{align}
&\quad 2\etaj\Bj(1 - 2\etaj L)\E\la \nabla f(\td{x}_{j}), \td{x}_{j} - x^{*}\ra\nonumber\\
&\le\E\|\td{x}_{j-1} - x^{*}\|^{2} - \E\|\td{x}_{j} - x^{*}\|^{2} + 4\etaj^{2}L\Bj \E (f(\td{x}_{j-1}) - f(x^{*})) - 4\etaj^{2}L\Bj \E (f(\td{x}_{j}) - f(x^{*}))\nonumber\\
&\quad  + 2\etaj^{2}\Bj^{2}\lb 1 + \frac{1}{\Bj}\rb\E \|\ej\|^{2} + 2\etaj^{2}\Bj \E \|\nabla f(\td{x}_{j})\|^{2} + \frac{1}{2} \E \|\td{x}_{j} - \td{x}_{j-1}\|^{2}\nonumber\\
&\le\E\|\td{x}_{j-1} - x^{*}\|^{2} - \E\|\td{x}_{j} - x^{*}\|^{2} + 4\etaj^{2}L\Bj \E (f(\td{x}_{j-1}) - f(x^{*})) - 4\etaj^{2}L\Bj \E (f(\td{x}_{j}) - f(x^{*}))\nonumber\\
& \quad + 2\etaj^{2}\Bj^{2}\lb 1 + \frac{1}{\Bj}\rb\E \|\ej\|^{2} + 2\etaj^{2}\Bj \E \|\nabla f(\td{x}_{j})\|^{2}\nonumber\\
& \quad  + 2\etaj \Bj (1 + 2\etaj L)\E (f(\td{x}_{j-1}) - f(x^{*})) - 2\etaj \Bj (1 + 2\etaj L) \E (f(\td{x}_{j}) - f(x^{*})) + 4\etaj^{2}L\Bj \E \la\nabla f(\td{x}_{j}), \td{x}_{j} - x^{*}\ra \nonumber\\
& \quad + 2 \etaj^{2}\Bj \E \|\nabla f(\td{x}_{j})\|^{2} + 2\etaj^{2}\Bj^{2}\lb 1 + \frac{1}{\Bj}\rb\E \|\ej\|^{2}\quad \mbox{(Corollary \ref{cor:dual_first})}\nonumber\\
& =  \E\|\td{x}_{j-1} - x^{*}\|^{2} - \E\|\td{x}_{j} - x^{*}\|^{2} + 2\etaj\Bj\lb 1+ 4\etaj L\rb \E (f(\td{x}_{j-1}) - f(x^{*}))- 2\etaj \Bj (1 + 4\eta L)\E (f(\td{x}_{j}) - f(x^{*}))\nonumber\\
& \quad  + 4\etaj^{2}\Bj^{2}\lb 1 + \frac{1}{\Bj}\rb\E \|\ej\|^{2} + 4\etaj^{2}L\Bj \E \la\nabla f(\td{x}_{j}), \td{x}_{j} - x^{*}\ra + 4\etaj^{2}\Bj \E \|\nabla f(\td{x}_{j})\|^{2} \nonumber\\
& =  \E\|\td{x}_{j-1} - x^{*}\|^{2} - \E\|\td{x}_{j} - x^{*}\|^{2} + 2\etaj\Bj\lb 1+ 4\etaj L\rb \E (f(\td{x}_{j-1}) - f(x^{*}))- 2\etaj \Bj (1 + 4\etaj L)\E (f(\td{x}_{j}) - f(x^{*}))\nonumber\\
& \quad  + 4\etaj^{2}\Bj\lb 1 + \frac{1}{\Bj}\rb \lb 4L \E (f(\td{x}_{j-1}) - f(x^{*})) + 2\H\rb\cdot I(\Bj < n) + 4\etaj^{2}L\Bj \E \la\nabla f(\td{x}_{j}), \td{x}_{j} - x^{*}\ra \nonumber\\
& \quad + 4\etaj^{2}\Bj \E \|\nabla f(\td{x}_{j})\|^{2}\qquad \mbox{(Lemma \ref{lem:ej})}\nonumber\\
& =  \E\|\td{x}_{j-1} - x^{*}\|^{2} - \E\|\td{x}_{j} - x^{*}\|^{2} + 2\etaj\Bj\lb 1+ 4\etaj L + 8\etaj L \lb 1 + \frac{1}{\Bj}\rb I(\Bj < n)\rb \E (f(\td{x}_{j-1}) - f(x^{*}))\nonumber\\
& \quad - 2\etaj \Bj (1 + 4\etaj L)\E (f(\td{x}_{j}) - f(x^{*}))   + 4\etaj^{2}L\Bj \E \la\nabla f(\td{x}_{j}), \td{x}_{j} - x^{*}\ra + 4\etaj^{2}\Bj \E \|\nabla f(\td{x}_{j})\|^{2}\nonumber\\
& \quad + 8\etaj^{2}\Bj\lb 1 + \frac{1}{\Bj}\rb \cdot \H\cdot I(\Bj < n) \label{eq:dual_second_2}
\end{align}  
The above equation \eqref{eq:dual_second_2} implies that 
\begin{align}
&\quad 2\etaj\Bj(1 - 4\etaj L)\E\la \nabla f(\td{x}_{j}), \td{x}_{j} - x^{*}\ra + \E\|\td{x}_{j} - x^{*}\|^{2}+ 2\etaj \Bj (1 + 4\etaj L)\E (f(\td{x}_{j}) - f(x^{*}))  \nonumber\\  
& \le  \E\|\td{x}_{j-1} - x^{*}\|^{2} + 2\etaj\Bj\lb 1+ 4\etaj L + 8\etaj L \lb 1 + \frac{1}{\Bj}\rb I(\Bj < n)\rb \E (f(\td{x}_{j-1}) - f(x^{*}))\nonumber\\
& \quad + 4\etaj^{2}\Bj\E \|\nabla f(\td{x}_{j})\|^{2} + 8\etaj^{2}\Bj\lb 1 + \frac{1}{\Bj}\rb \cdot \H\cdot I(\Bj < n)\label{eq:dual_second_3}
\end{align}
By Lemma \ref{lem:cocoercive} with $g = f, x = x^{*}, y = \td{x}_{j}$, 
\[\|\nabla f(\td{x}_{j})\|^{2}\le 2L\lb f(x^{*}) - f(\td{x}_{j}) + \la \nabla f(\td{x}_{j}), \td{x}_{j} - x^{*}\ra\rb.\]
This together with \eqref{eq:dual_second_3} imply that 
\begin{align}
&\quad 2\etaj\Bj(1 - 8\etaj L)\E\la \nabla f(\td{x}_{j}), \td{x}_{j} - x^{*}\ra + \E\|\td{x}_{j} - x^{*}\|^{2}+ 2\etaj \Bj(1 + 8\etaj L) \E (f(\td{x}_{j}) - f(x^{*}))  \nonumber\\  
& \le  \E\|\td{x}_{j-1} - x^{*}\|^{2} + 2\etaj\Bj\lb 1+ 4\etaj L + 8\etaj L \lb 1 + \frac{1}{\Bj}\rb I(\Bj < n)\rb \E (f(\td{x}_{j-1}) - f(x^{*}))\nonumber\\
& \quad  + 8\etaj^{2}\Bj\lb 1 + \frac{1}{\Bj}\rb \cdot \H\cdot I(\Bj < n).\label{eq:dual_second_4}  
\end{align}
Since $\etaj L \le \frac{1}{8}$, the assumption \textbf{A}2 implies that 
\[2\etaj\Bj(1 - 8\etaj L)\E\la \nabla f(\td{x}_{j}), \td{x}_{j} - x^{*}\ra \ge 2\etaj\Bj(1 - 8\etaj L) \lb\E (f(\td{x}_{j}) - f(x^{*})) + \frac{\mu}{2}\E\|\td{x}_{j} - x^{*}\|^{2}\rb.\]
Therefore,
\begin{align}
&\quad \lb 1 + \mu\etaj\Bj (1 - 8\etaj L)\rb\E\|\td{x}_{j} - x^{*}\|^{2}+ 4\etaj \Bj \E (f(\td{x}_{j}) - f(x^{*}))  \nonumber\\  
& \le  \E\|\td{x}_{j-1} - x^{*}\|^{2} + 2\etaj\Bj\lb 1+ 4\etaj L + 8\etaj L \lb 1 + \frac{1}{\Bj}\rb I(\Bj < n)\rb \E (f(\td{x}_{j-1}) - f(x^{*}))\nonumber\\
& \quad  + 8\etaj^{2}\Bj\lb 1 + \frac{1}{\Bj}\rb \cdot \H\cdot I(\Bj < n).\label{eq:dual_second_5}  
\end{align}
Finally, since $\Bj \ge 8$, we conclude that 
\begin{align*}
&\quad \lb 1 + \mu\etaj\Bj (1 - 8\etaj L)\rb\E\|\td{x}_{j} - x^{*}\|^{2}+ 4\etaj \Bj \E (f(\td{x}_{j}) - f(x^{*}))  \\  
& \le  \E\|\td{x}_{j-1} - x^{*}\|^{2} + 2\etaj\Bj\lb 1+ 4\etaj L + 9\etaj L \cdot  I(\Bj < n)\rb \E (f(\td{x}_{j-1}) - f(x^{*})) + 9\etaj^{2}\Bj \cdot \H\cdot I(\Bj < n).
\end{align*}

\end{proof}

\begin{proof}[\textbf{of Lemma \ref{lem:geom_finite}}]
  We prove the first claim by induction. When $j = 0$, the claim is obvious. Suppose we prove the claim for $j - 1$, i.e. 
\[\E\|\xj_{0} - x^{*}\|_{2}^{2} = \E \|\td{x}_{j-1} - x^{*}\|_{2}^{2} < \infty.\]
Let $\yj_{k}$ be another sequence constructed as follows:
\[\yj_{0} = \xj_{0}, \quad \yj_{k} = \yj_{k-1} - \etaj\zetaj_{k}, \quad \mbox{where }\zetaj = \nabla f_{i_{k-1}}(\yj_{k-1}) - \nabla f_{i_{k-1}}(\yj_{0}) + \nabla f(\yj_{0}).\]
In other words, $\yj_{k}$ is a hypothetical sequence of iterates produced by SVRG initialized at $\xj_{0}$ and updated using the same sequence of random subsets. Let $\mathrm{Id}$ denote the identity mapping. Then 
\[\xj_{k} - \yj_{k} = \lb\mathrm{Id} - \etaj \nabla f_{i_{k-1}}\rb(\xj_{k-1}) - \lb\mathrm{Id} - \etaj \nabla f_{i_{k-1}}\rb(\yj_{k-1}) - \etaj \ej.\]
where we use the fact that $\mu_{j} = \nabla f(\xj_{0}) + \ej = \nabla f(\yj_{0}) + \ej$. Since $f_{i_{k-1}}$ is $L$-smooth and convex and $\etaj \le 1 / 2L$, it is well known that $\mathrm{Id} - \etaj \nabla f_{i_{k-1}}$ is a non-expansive operator. Thus,
\begin{align*}
  \|\xj_{k} - \yj_{k}\|_{2}&\le \|\lb\mathrm{Id} - \etaj \nabla f_{i_{k-1}}\rb(\xj_{k-1}) - \lb\mathrm{Id} - \etaj \nabla f_{i_{k-1}}\rb(\yj_{k-1})\|_{2} + \etaj \|\ej\|_{2}\\
& \le \|\xj_{k-1} - \yj_{k-1}\|_{2} + \etaj\|\ej\|_{2}.
\end{align*}
As a result,
\begin{equation}
  \label{eq:xjkyjk}
  \|\xj_{k} - \yj_{k}\|_{2}\le \|\xj_{0} - \yj_{0}\|_{2} + \etaj k \|\ej\|_{2} = \etaj k \|\ej\|_{2}.
\end{equation}
On the other hand, \cite{SVRG} showed in the proof of their Theorem 1 that
\begin{align}
\lefteqn{2\etaj(1 - 2\etaj L) \E (f(\yj_{k}) - f(x^{*})) + \E \|\yj_{k+1} - x^{*}\|_{2}^{2}}\nonumber \\
& \le 4\eta^{2}L \E (f(\yj_{0}) - f(x^{*})) + \E \|\yj_{k} - x^{*}\|_{2}^{2}.\nonumber
\end{align}
Since $f(\yj_{k}) - f(x^{*})\ge 0$ and $1 - 2\etaj L\ge 0$, we have
\[\E \|\yj_{k+1} - x^{*}\|_{2}^{2} \le 4\eta^{2}L \E (f(\yj_{0}) - f(x^{*})) + \E \|\yj_{k} - x^{*}\|_{2}^{2}.\]
As a result, 
\begin{align}
  &\E \|\yj_{k} - x^{*}\|_{2}^{2}\le 4k\eta^{2}L \E (f(\yj_{0}) - f(x^{*})) + \E\|\yj_{0} - x^{*}\|_{2}^{2}\nonumber\\
&\le (4k\etaj^{2} L^{2} + 1)\E\|\yj_{0} - x^{*}\|_{2}^{2} \le (k + 1)\E\|\xj_{0} - x^{*}\|_{2}^{2}.  \label{eq:yjkxstar}
\end{align}
Putting \eqref{eq:xjkyjk} and \eqref{eq:yjkxstar} together, and using the fact that $\|a + b\|_{2}^{2}\le 2\|a\|_{2}^{2} + 2\|b\|_{2}^{2}$, we obtain that
\begin{align}
  \E\|\xj_{k} - x^{*}\|_{2}^{2} &\le 2 \E \|\xj_{k} - \yj_{k}\|_{2}^{2} + 2 \E \|\yj_{k} - x^{*}\|_{2}^{2}\nonumber\\
& \le k^{2}\lb \etaj^{2} \E \|\ej\|_{2}^{2} + \E\|\xj_{0} - x^{*}\|_{2}^{2}\rb.\label{eq:xjkxstar}
\end{align}
By Lemma \ref{lem:ej}, 
\begin{equation}
  \label{eq:ejfinite_L2}
  \E \|\ej\|_{2}^{2}\le 4L\E (f(\xj_{0}) - f(x^{*})) + 2\H\le 2L^{2} \E\|\xj_{0} - x^{*}\|_{2}^{2} + 2\H.
\end{equation}
By \eqref{eq:xjkxstar},
\[\E\|\xj_{k} - x^{*}\|_{2}^{2}\le 2k^{2}\lb \E\|\xj_{0} - x^{*}\|_{2}^{2} + \etaj^{2}\H\rb.\]
By the induction hypothesis,
\[\E\|\td{x}_{j} - x^{*}\|_{2}^{2} \le 2\E \Nj^{2} \lb \E\|\xj_{0} - x^{*}\|_{2}^{2} + \etaj^{2}\H\rb < \infty,\]
and
\[\E \|\ej\|_{2}^{2} < \infty.\]
By Corollary \ref{cor:nuknorm2}, 
\begin{align*}
  \E\|\nuj_{k}\|_{2}^{2} &\le 4L(f(x^{*}) - f(\xj_{k}) - \la\nabla f(\xj_{k}), x^{*} - \xj_{k}\ra) + 4L \lb f(\xj_{0}) - f(x^{*})\rb\\
& \qquad  + 2 \E \|\nabla f(\xj_{k})\|^{2} + 2\E\|\ej\|_{2}^{2}\\
& \le 2L^{2}\|\xj_{k} - x^{*}\|_{2}^{2} + 2L^{2}\|\xj_{0} - x^{*}\|_{2}^{2} + 2 \E \|\nabla f(\xj_{k})\|^{2} + 2\E\|\ej\|_{2}^{2}\\
& \le 4L^{2}\|\xj_{k} - x^{*}\|_{2}^{2} + 2L^{2}\|\xj_{0} - x^{*}\|_{2}^{2} + 2\E\|\ej\|_{2}^{2}\\
\end{align*}
Then the first claim yields
\[\E \|\nuj_{\Nj}\|_{2}^{2}\le 4L^{2} \E \|\td{x}_{j} - x^{*}\|_{2}^{2} + 2L^{2} \E \|\td{x}_{j-1} - x^{*}\|_{2}^{2} + 2\E\|\ej\|_{2}^{2} < \infty.\]
Finally, 
\begin{align*}
  &\E |\la \ej, \td{x}_{j} - \td{x}_{j-1}\ra|\le \frac{1}{2}\E \|\ej\|_{2}^{2} + \frac{1}{2}\E \|\td{x}_{j} - \td{x}_{j-1}\|_{2}^{2} \\
& \le \frac{1}{2}\E \|\ej\|_{2}^{2} + \E \|\td{x}_{j} - x^{*}\|_{2}^{2} + \E\|\td{x}_{j-1} - x^{*}\|_{2}^{2} < \infty.
\end{align*}
\end{proof}

\section{Analysis of R-SVRG}
Throughout the rest of appendices, we will denote $T(\eps)$ and $T_{x}(\eps)$ by
\begin{equation}
  \label{eq:Teps}
  T(\eps) = \min\{T': \E (f(\td{x}_{j}) - f(x^{*}))\le \eps, \,\,\forall T \ge T'\},
\end{equation}
and 
\begin{equation}
  \label{eq:Txeps}
  T_{x}(\eps) = \min\{T': \E \|\td{x}_{j} - x^{*}\|^{2}\le \eps, \,\,\forall T \ge T'\}.
\end{equation}
Then we have
\begin{equation}
  \label{eq:comp_Teps}
  \E\comp(\eps) = 2n T(\eps), \quad \E \comp_{x}(\eps) = 2nT_{x}(\eps).
\end{equation}

~\\
\noindent Although we can directly apply Theorem \ref{thm:dual_second} with $B = n$, the constants involved in the analysis are compromised. To sharpen the constants, we derive a counterpart of Theorem \ref{thm:dual_second} for R-SVRG.
\begin{theorem}\label{thm:SVRG}
Let $B =  n$ and assume that $\eta L \le \frac{1}{3}$. Under the assumption \textbf{A}1 and \textbf{A}2, 
\begin{align*}
 & \lb 1 + \mu\etaj  n (1 - 3\etaj L)\rb\E \|\td{x}_{j} - x^{*}\|^{2} + 2\etaj  n \E (f(\td{x}_{j}) - f(x^{*}))\\
& \le \E\|\td{x}_{j-1} - x^{*}\|^{2} + 4\etaj^{2}L n  \E (f(\td{x}_{j-1}) - f(x^{*})) 
\end{align*}
\end{theorem}
\begin{proof}
  In this case, $\ej \equiv 0$. By Lemma \ref{lem:dual} with $u = x^{*}$ and Lemma \ref{lem:nuknorm}, 
  \begin{align*}
&2\etaj n \E\la \nabla f(\td{x}_{j}), \td{x}_{j} - x^{*} \ra \le\E\|\td{x}_{j-1} - x^{*}\|^{2} - \E\|\td{x}_{j} - x^{*}\|^{2}\\
+ &\etaj^{2} n \lb -4L \E(f(\td{x}_{j}) - f(x^{*})) + 4L \E (f(\td{x}_{j-1}) - f(x^{*})) + 4L \E \la \nabla f(\td{x}_{j}), \td{x}_{j} - x^{*}\ra + \E \|\nabla f(\td{x}_{j})\|^{2}\rb.
  \end{align*}
By Lemma \ref{lem:cocoercive} with $g = f, x = x^{*}, y = \td{x}_{j}$, 
\[\|\nabla f(\td{x}_{j})\|^{2}\le 2L\lb f(x^{*}) - f(\td{x}_{j}) + \la\nabla f(\td{x}_{j}), \td{x}_{j} -x^{*}\ra\rb.\]
Thus, 
  \begin{align*}
&2\etaj n \E\la \nabla f(\td{x}_{j}), \td{x}_{j} - x^{*} \ra \le\E\|\td{x}_{j-1} - x^{*}\|^{2} - \E\|\td{x}_{j} - x^{*}\|^{2}\\
+ &\etaj^{2} n \lb -6L \E(f(\td{x}_{j}) - f(x^{*})) + 4L \E (f(\td{x}_{j-1}) - f(x^{*})) + 6L \E \la \nabla f(\td{x}_{j}), \td{x}_{j} - x^{*}\ra \rb.
\end{align*}
Rearranging the terms we obtain that 
\begin{align*}
&2\etaj n (1 - 3\etaj L)\E\la \nabla f(\td{x}_{j}), \td{x}_{j} - x^{*} \ra \le\E\|\td{x}_{j-1} - x^{*}\|^{2} - \E\|\td{x}_{j} - x^{*}\|^{2}\\
& \quad +\etaj^{2} n \lb -6L \E(f(\td{x}_{j}) - f(x^{*})) + 4L \E (f(\td{x}_{j-1}) - f(x^{*}))\rb.
  \end{align*}
Since $\etaj L \le \frac{1}{3}$, the assumption \textbf{A}2 implies that 
\[2\etaj n (1 - 3\etaj L)\E\la \nabla f(\td{x}_{j}), \td{x}_{j} - x^{*} \ra\ge 2\etaj n (1 - 3\etaj L)\lb \E (f(\td{x}_{j}) - f(x^{*})) + \frac{\mu}{2}\|\td{x}_{j} - x^{*}\|^{2}\rb.\]
This entails that 
\begin{align*}
 & \lb 1 + \mu\etaj  n (1 - 3\etaj L)\rb\E \|\td{x}_{j} - x^{*}\|^{2} + 2\etaj  n \E (f(\td{x}_{j}) - f(x^{*}))\\
& \le \E\|\td{x}_{j-1} - x^{*}\|^{2} + 4\etaj^{2}L n  \E (f(\td{x}_{j-1}) - f(x^{*})) 
\end{align*}
\end{proof}

Based on Theorem \ref{thm:SVRG}, we can prove the results in Section \ref{sec:RSVRG}. 

\begin{proof}[\textbf{Theorem \ref{thm:SVRG_bound}}]
\begin{enumerate}[(1)]
 \item  By Theorem \ref{thm:SVRG}, we have
\[2\etaj n \E (f(\td{x}_{j}) - f(x^{*}))\le 4\etaj^{2}L n  \E (f(\td{x}_{j-1}) - f(x^{*})) + \E \|\td{x}_{j-1} - x^{*}\|^{2} - \E \|\td{x}_{j} - x^{*}\|^{2}.\]
Summing the above inequality for $j = 1, \ldots, T$, we have
\[2\etaj  n  (1 - 2\etaj L)\sum_{j=1}^{T}\E (f(\td{x}_{j}) - f(x^{*})) \le 4\etaj^{2}L  n  \Delta_{f} + \Delta_{x}.\]
By convexity, 
\[\sum_{j=1}^{T}\E (f(\td{x}_{j}) - f(x^{*}))\ge T\lb\E (f(\bar{x}_{T}) - f(x^{*}))\rb.\]
Therefore,
\[\E (f(\bar{x}_{T}) - f(x^{*}))\le \frac{1}{T}\cdot\frac{4\etaj^{2}L  n  \Delta_{f} + \Delta_{x}}{2\etaj n (1 - 2\etaj L)}.\]
\item Let $Q_{j} = (1 + \mu\etaj n (1 - 3\etaj L))\E \|\td{x}_{j} - x^{*}\|^{2} + 2\etaj n\E (f(\td{x}_{j}) - f(x^{*}))$, then by Theorem \ref{thm:SVRG},
\[Q_{j}\le \max\left\{2\etaj L, \frac{1}{1 + \mu\etaj n(1 - 3\etaj L)}\right\}\cdot Q_{j-1} = \lambda Q_{j-1}.\]
This implies that 
\[Q_{T}\le \lambda^{T}Q_{0}.\]
The result is then proved by noticing that 
\[Q_{T}\ge \E \|\td{x}_{T} - x^{*}\|^{2} + 2\etaj n \E (f(\bar{x}_{T}) - f(x^{*}))\]
and 
\[Q_{0} \le \E \|\td{x}_{0} - x^{*}\|^{2} + 2\etaj n (2 - 3\etaj L)\E (f(\td{x}_{0}) - f(x^{*}))\le \E \|\td{x}_{0} - x^{*}\|^{2} + 4\etaj n \E (f(\td{x}_{0}) - f(x^{*})).\]
\end{enumerate}
\end{proof}

\begin{proof}[\textbf{Corollary \ref{cor:SVRG_complexity}}]
In the non-strongly convex case, by part (1) of Theorem \ref{thm:SVRG_bound},
\[\E (f(\bar{x}_{T}) - f(x^{*}))\le \frac{1}{T}\cdot\frac{4(\etaj L)^{2}  n  \Delta_{f} + L\Delta_{x}}{2(\etaj L) n (1 - 2\etaj L)} = O\lb\frac{\Delta_{f}}{T} + \frac{L\Delta_{x}}{Tn}\rb.\]
Recalling the definitions of $T(\eps)$ and $T_{x}(\eps)$ in \eqref{eq:Teps} and \eqref{eq:Txeps}, this implies that 
\[T(\eps) = O\lb \frac{\Delta_{f}}{\eps} + \frac{L\Delta_{x}}{n\eps}\rb,\]
and hence
\[\E \comp(\eps) = O\lb n T(\eps)\rb = O\lb \frac{n\Delta_{f} + L\Delta_{x}}{\eps}\rb.\]
In the strongly convex case, by part (2) of Theorem \ref{thm:SVRG_bound},
\[\E \|\td{x}_{j} - x^{*}\|^{2} + 2\etaj n \E(f(\td{x}_{j}) - f(x^{*}))\le  \lambda^{T}Q_{0}.\]
This implies that 
\[T(\eps) = O\lb \log \frac{LQ_{0}}{n\eps} \bigg/ \log \frac{1}{\lambda}\rb, \mbox{ and }T_{x}(\eps) = O\lb \log \frac{Q_{0}}{\eps} \bigg/ \log \frac{1}{\lambda}\rb\]
By definition,
\[\log \frac{1}{\lambda} = \log \lb \lb 1 + \frac{n\theta (1 - 3\theta)}{\kappa}\rb \wedge \frac{3}{2}\rb.\]
This implies that 
\[\lb\log \frac{1}{\lambda}\rb^{-1} = O\lb \frac{\kappa}{n} + 1\rb.\]
Note that $LQ_{0} = L\Delta_{x} + 2\theta n\Delta_{f}$, we obtain that
\begin{align*}
&T(\eps) = O\lb \lb\frac{\kappa}{n} + 1\rb\log \lb\frac{n\Delta_{f} + L\Delta_{x}}{n\eps}\rb \rb\\
\Longrightarrow &\E \comp(\eps) = O\lb n T(\eps)\rb = O\lb (n + \kappa) \log \lb\frac{n\Delta_{f} + L\Delta_{x}}{n\eps}\rb\rb.
\end{align*}
Similarly, 
\begin{align*}
&T_{x}(\eps) = O\lb \lb\frac{\kappa}{n} + 1\rb\log \lb\frac{n\Delta_{f} + L\Delta_{x}}{L\eps}\rb \rb\\
\Longrightarrow &\E \comp_{x}(\eps) = O\lb n T_{x}(\eps)\rb = O\lb (n + \kappa) \log \lb\frac{n\Delta_{f} + L\Delta_{x}}{L\eps}\rb\rb.
\end{align*}

\end{proof}

\begin{proof}[\textbf{Corollary \ref{cor:SVRG_complexity_ns}}]
By part (1) of Theorem \ref{thm:SVRG_bound}, we have
\[\E (f(\bar{x}_{T}) - f(x^{*}))\le \frac{1}{T}\cdot\frac{4(\etaj L)^{2}  n  \Delta_{f} + L\Delta_{x}}{2(\etaj L) n (1 - 2\etaj L)} = \frac{1}{T}\cdot \frac{4\theta^{2} \Delta_{f} + L\Delta_{x}}{\sqrt{n} \cdot 2\theta (1 - 2\theta)} = O\lb\frac{\Delta_{f} + L\Delta_{x}}{T \sqrt{n}}\rb.\]
The assumption \textbf{A}1, 
\[\Delta_{f}\le \frac{L}{2}\Delta_{x}.\]
Thus, 
\[\E (f(\bar{x}_{T}) - f(x^{*})) = O\lb\frac{L\Delta_{x}}{T \sqrt{n}}\rb.\]
This implies that 
\[T(\eps) = O\lb\left\lceil \frac{L\Delta_{x}}{\sqrt{n}\eps}\right\rceil\rb = O\lb 1 + \frac{L\Delta_{x}}{\sqrt{n}\eps}\rb,\]
and hence
\[\E \comp(\eps) = O(nT(\eps)) = O\lb n + \frac{\sqrt{n}L\Delta_{x}}{\eps}\rb.\]
\end{proof}

\section{Analysis of SCSG}

\subsection{Convergence Analysis for Non-Strongly Convex Objectives}

\begin{proof}[\textbf{Theorem \ref{thm:SCSG_bound}}]
By Theorem \ref{thm:dual_second}, we have
   \begin{align*}
&4\etaj B \E (f(\td{x}_{j}) - f(x^{*}))\\
\le &2\etaj B(1 + 13\etaj L) \E (f(\td{x}_{j-1}) - f(x^{*})) + \E \|\td{x}_{j-1} - x^{*}\|^{2} - \E \|\td{x}_{j} - x^{*}\|^{2} + 9\etaj^{2} B \H \cdot I(B < n).
\end{align*}
Telescoping the above inequality for $j = 1, \ldots, T$, we have
\begin{align*}
2\etaj  B  (1 - 13\etaj L)\sum_{j=1}^{T}\E (f(\td{x}_{j}) - f(x^{*})) &\le 2\etaj B(1 + 13\etaj L ) \Delta_{f} + \Delta_{x} + 9\etaj^{2} B T \H \cdot I(B < n)\\
&\le 4\etaj B \Delta_{f} + \Delta_{x} + 9\etaj^{2} B T \H \cdot I(B < n),
\end{align*}
where the last inequality uses $13\etaj L\le 1$. By convexity, 
\[\sum_{j=1}^{T}\E (f(\td{x}_{j}) - f(x^{*}))\ge T\lb\E (f(\bar{x}_{T}) - f(x^{*}))\rb.\]
Therefore,
\[\E (f(\bar{x}_{T}) - f(x^{*}))\le \frac{1}{T}\cdot\frac{4\etaj B \Delta_{f} + \Delta_{x}}{2\etaj B (1 - 13\etaj L)} + \frac{9\etaj\H\cdot I(B < n)}{2(1 - 13\etaj L)}.\]
\end{proof}

Before proving the results in Section \ref{sec:ana_SCSG}, we derive the computation complexity for arbitrary batch size $B$ with an appropriately scaled stepsize $\eta$ in the non-strongly convex case.

\begin{corollary}\label{cor:SCSG_complexity_general}
Assume \textbf{A}1 holds. Set  $\etaj = \frac{\alpha}{L }$ with
\[\frac{\alpha}{1 - 13\alpha} \cdot \frac{9\H\cdot I(B < n)}{2L}< \eps, \] 
then with the output $\bar{x}_{T}$,
\[\E \comp(\eps) = O\lb \frac{B\Delta_{f}}{\eps} + \frac{L\Delta_{x}}{\eps \alpha}\rb.\]
\end{corollary}
\begin{proof}
Let 
\[\phi = \frac{\alpha}{1 - 13\alpha}\cdot \frac{9\H\cdot I(B < n)}{2L\eps}.\]
By part (1) of Theorem \ref{thm:SCSG_bound},  
\begin{equation}\label{eq:SCSG_complexity}
\E (f(\bar{x}_{T}) - f(x^{*}))\le \frac{1}{T}\cdot\frac{4\etaj B \Delta_{f} + \Delta_{x}}{2\etaj B (1 - 13\etaj L)} + \frac{9\etaj\H \cdot I(\Bj < n)}{2(1 - 13\etaj L)}\triangleq \frac{1}{T}D_{1} + D_{2}
\end{equation}
Under these conditions, $D_{2}$ is bounded by
\[D_{2} = \frac{\alpha}{1 - 13\alpha}\cdot \frac{9\H \cdot I(B < n)}{2L} = \phi \eps < \eps.\]
Let 
\[\td{T}(\eps) = \frac{D_{1}}{(1 - \phi)\eps}\]
then for any $T \ge \td{T}(\eps)$, 
\[\E (f(\bar{x}_{T}) - f(x^{*}))\le (1 - \phi)\eps + \phi\eps = \eps.\]
This implies that 
\[T(\eps)\le \td{T}(\eps) = O\lb\frac{D_{1}}{\eps}\rb.\]
By definition, 
\[D_{1} = O\lb \Delta_{f} + \frac{\Delta_{x}}{\etaj B}\rb = O\lb \Delta_{f} + \frac{L\Delta_{x}}{\alpha B}\rb.\]
Therefore,
\[\E \comp(\eps) = O\lb B \td{T}(\eps)\rb = O\lb\frac{BD_{1}}{\eps}\rb = O\lb \frac{B\Delta_{f}}{\eps} + \frac{L\Delta_{x}}{\eps \alpha}\rb.\]
\end{proof}

\begin{proof}[\textbf{Corollary \ref{cor:SCSG_complexity}}]
Let $\alpha = \theta / B$. Then 
\[\frac{\alpha}{1 - 13\alpha} \cdot \frac{9\H\cdot I(B < n)}{2L}= \frac{\theta}{1 - 13\theta / B}\cdot \frac{9\H \cdot I(B < n)}{2L B}\le \frac{\theta}{1 - 13\theta / B}\cdot \frac{9\eps}{2\gamma} < \eps.\]
By Corollary \ref{cor:SCSG_complexity_general}, 
\[\E \comp(\eps) = O\lb \frac{B\Delta_{f}}{\eps} + \frac{L\Delta_{x}}{\eps}\cdot \frac{B}{\theta}\rb = O \lb \frac{BL\Delta_{x}}{\eps}\rb,\]
where the last equality uses the fact that $\Delta_{f}\le \frac{L\Delta_{x}}{2}$ and $\theta = \Theta(1)$. As a consequence, 
\[\E \comp(\eps) = O\lb \frac{\H \Delta_{x}}{\eps^{2}}\wedge \frac{nL\Delta_{x}}{\eps}\rb\]
\end{proof}

\subsection{Convergence Analysis for Strongly Convex Objectives}
\begin{proof}[\textbf{Theorem \ref{thm:SCSG_bound_sc}}]
For convenience, let 
\[\Delta_{e} = 9\etaj^{2}\H \cdot I(B < n), \quad \xi = \mu\etaj B, \quad \alpha = L\etaj\]
By Theorem \ref{thm:dual_second}, we have
\begin{align*}
&(1 + \xi(1 - 8\alpha))\E \|\td{x}_{j} - x^{*}\|^{2} + \frac{4\xi}{\mu}\E (f(\td{x}_{j}) - f(x^{*}))\\
\le &\E \|\td{x}_{j-1} - x^{*}\|^{2} + \frac{2(1 + 13\alpha)\xi}{\mu}\E (f(\td{x}_{j-1}) - f(x^{*})) + \Delta_{e}.
\end{align*}
The assumption \textbf{A}2 implies that $f(\td{x}_{j}) - f(x^{*})\ge \frac{\mu}{2}\|\td{x}_{j} - x^{*}\|^{2}$. Thus,
\begin{align}
&(1 + \xi)\E \|\td{x}_{j} - x^{*}\|^{2} + \frac{2\lb 2 - 8\alpha\rb\xi}{\mu} \E (f(\td{x}_{j}) - f(x^{*}))\nonumber\\
\le &(1 + \xi(1 - 8\alpha))\E \|\td{x}_{j} - x^{*}\|^{2} + \frac{4\xi}{\mu}\E (f(\td{x}_{j}) - f(x^{*}))\nonumber\\
\le & \E \|\td{x}_{j-1} - x^{*}\|^{2} + \frac{2(1 + 13\alpha)\xi}{\mu}\E (f(\td{x}_{j-1}) - f(x^{*})) + \Delta_{e}.\label{eq:SCSG_bound_sc_1}
\end{align}
On the other hand, 
\begin{align}
  \frac{2\lb 2 - 8\alpha\rb\xi}{\mu} & = \frac{2\lb(1 + 13\alpha)(1 + \xi) + \phi\rb\xi}{\mu}\ge \frac{2 (1 + 13\alpha)(1 + \xi)\xi}{\mu}.  \label{eq:SCSG_bound_sc_2}
\end{align}
Putting \eqref{eq:SCSG_bound_sc_1} and \eqref{eq:SCSG_bound_sc_2}, we obtain that 
\begin{align}
&(1 + \xi)\left\{\E \|\td{x}_{j} - x^{*}\|^{2} + \frac{2\lb 1 + 13\alpha\rb\xi}{\mu} \E (f(\td{x}_{j}) - f(x^{*}))\right\}\nonumber\\
\le & \E \|\td{x}_{j-1} - x^{*}\|^{2} + \frac{2(1 + 13\alpha)\xi}{\mu}\E (f(\td{x}_{j-1}) - f(x^{*})) + \Delta_{e}.\label{eq:SCSG_bound_sc_3}
\end{align}
Multiplying both sides of \eqref{eq:SCSG_bound_sc_1} by $(1 + \xi)^{j-1}$ and summing over $j = T, T - 1, \ldots, 1$, we obtain that 
\begin{align}
&(1 + \xi)^{T}\left\{\E \|\td{x}_{T} - x^{*}\|^{2} + 2\etaj B \E (f(\td{x}_{T}) - f(x^{*}))\right\} \nonumber\\ 
\le &(1 + \xi)^{T}\left\{\E \|\td{x}_{T} - x^{*}\|^{2} + \frac{2\lb 1 + 13\alpha\rb\xi}{\mu} \E (f(\td{x}_{T}) - f(x^{*}))\right\} \nonumber\\ 
\le & \Delta_{x} + \frac{2\lb 1 + 13\alpha\rb\xi}{\mu}\Delta_{f}+ \Delta_{e}\sum_{j=1}^{T}(1 + \xi)^{j-1}.\nonumber\\
\le &\frac{2}{\mu}\Delta_{f} + \frac{2(1 - 21\alpha - \phi)}{\mu}\Delta_{f} + \Delta_{e}\frac{(1 + \xi)^{T}}{\xi}\nonumber\\
\le & \frac{4}{\mu}\Delta_{f} + \Delta_{e}\frac{(1 + \xi)^{T}}{\xi}.\label{eq:SCSG_bound_sc_4}
\end{align}
\end{proof}

\begin{proof}[\textbf{Corollary \ref{cor:SCSG_complexity_comp_last}}]
By definition, $B \ge \gamma \kappa$, and
  \[\mu\etaj B \le \theta, \quad \etaj L \le \frac{\theta\kappa}{B} \le \frac{\theta}{\gamma}.\]
As a result,
\[\phi \ge 2 - \frac{8\theta}{\gamma} - \lb 1 + \frac{13\theta}{\gamma}\rb(1 + \theta) > 2 - \frac{8}{22} - \lb1 + \frac{13}{22}\rb\lb1 + \frac{2}{9}\cdot \frac{1}{22}\rb> 0.029 > 0.\]
By \eqref{eq:SCSG_sc_F} in Theorem \ref{thm:SCSG_bound_sc}, 
\begin{equation}\label{eq:SCSG_complexity_comp_last}
\E (f(\td{x}_{T}) - f(x^{*}))\le \frac{2\Delta_{f}}{\mu\etaj B(1 + \mu\etaj B)^{T}} + \frac{9\H \cdot I(B < n)}{2\mu B}\triangleq \frac{D_{1}}{(1 + \xi)^{T}} + D_{2}.
\end{equation}
where $\xi = \mu\etaj B$. Then 
\[D_{2}\le \frac{9\eps}{2\gamma} < \eps.\]
Define $\td{T}(\eps)$ as
\[\td{T}(\eps) = \log \lb \lb 1 - \frac{9}{2\gamma}\rb\frac{2\Delta_{f}}{\eps \xi}\rb\bigg/ \log \lb1 + \xi\rb\]
then for any $T\ge \td{T}(\eps)$, 
\[\E (f(\td{x}_{T}) - f(x^{*}))\le \lb 1 - \frac{9}{2\gamma}\rb\eps + \frac{9\eps}{2\gamma} = \eps.\]
Therefore, 
\[T(\eps) \le \td{T}(\eps) = O\lb\log \lb\frac{\Delta_{f}}{\xi\eps}\rb\bigg/ (\xi\wedge 1)\rb.\]
Note that 
\[\xi = \theta\cdot \frac{B}{B \vee \gamma \kappa} \Longrightarrow  \frac{1}{\xi \wedge 1} = O\lb 1 + \frac{\kappa}{B}\rb,\]
and whenever $B < n$, $B\ge \gamma \kappa > \kappa$. Thus, 
\[\frac{1}{\xi \wedge 1} = O\lb 1 + \frac{\kappa}{n}\rb = O\lb \frac{n\vee \kappa}{n}\rb\]
Therefore,
\[\E \comp(\eps) = O\lb BT(\eps)\rb = O\lb\lb B + \kappa\rb\log \lb\frac{\Delta_{f}}{\eps}\cdot\frac{n\vee \kappa}{n}\rb\rb\] 
\[= O\lb \lb\frac{\H}{\mu \eps}\wedge n +  \kappa\rb\log \lb\frac{\Delta_{f}}{\eps}\cdot\frac{n\vee \kappa}{n}\rb\rb.\]

\end{proof}

\begin{proof}[\textbf{Corollary \ref{cor:SCSG_complexity_comp_last_iter}}]
Using the same argument as in the proof of Corollary \ref{cor:SCSG_complexity_comp_last}, $\phi > 0$.
By \eqref{eq:SCSG_sc_X} in Theorem \ref{thm:SCSG_bound_sc}, 
\begin{equation}\label{eq:SCSG_complexity_comp_last}
\E \|\td{x}_{T} - x^{*}\|^{2}\le \frac{2\Delta_{f}}{\mu (1 + \mu\etaj B)^{T}} + \frac{9\etaj\H \cdot I(B < n)}{\mu}\triangleq \frac{D_{1}}{(1 + \xi)^{T}} + D_{2}.
\end{equation}
where $\xi = \mu\etaj B$. Then 
\[D_{2}\le \frac{9\theta}{\mu^{2}B}\le  \frac{9\theta\eps}{\gamma} < \eps.\]
Define $\td{T}(\eps)$ as
\[\td{T}(\eps) = \log \lb \lb 1 - \frac{9\theta}{\gamma}\rb\frac{2\Delta_{f}}{\mu\eps}\rb\bigg/ \log \lb1 + \xi\rb\]
then for any $T\ge \td{T}(\eps)$, 
\[\E (f(\td{x}_{T}) - f(x^{*}))\le \lb 1 - \frac{9}{2\gamma}\rb\eps + \frac{9\eps}{2\gamma} = \eps.\]
Therefore, 
\[T(\eps) \le \td{T}(\eps) = O\lb\log \lb\frac{\Delta_{f}}{\mu\eps}\rb\bigg/ (\xi\wedge 1)\rb.\]
Similar to the proof of Corollary \ref{cor:SCSG_complexity_comp_last}, 
\[\E \comp(\eps) = O\lb \lb\frac{\H}{\mu^{2} \eps}\wedge n +  \kappa\rb\log \lb\frac{\Delta_{f}}{\mu\eps}\rb\rb.\]

\end{proof}

\section{Proof of Results in Section \ref{sec:Hf}}\label{app:section4}

\begin{proof}[\textbf{Proposition \ref{prop:general_H}}]
By Lemma \ref{lem:cocoercive} in Appendix \ref{app:lemmas} 
\[f_{i}(x) - f_{i}(x^{*})\ge \la \nabla f_{i}(x^{*}), x - x^{*}\ra + \frac{1}{2L}\|\nabla f_{i}(x^{*}) - \nabla f_{i}(x)\|^{2}.\]
Averaging the above inequality for all $i$ results in
\begin{equation}\label{eq:temp_coco}
f(x) \ge f(x^{*}) + \frac{1}{2nL}\sum_{i=1}^{n}\|\nabla f_{i}(x^{*}) - \nabla f_{i}(x)\|^{2}.
\end{equation}
Noticing that for any $a, b\in \R^{p}$, 
\[\|a - b\|^{2} = \|a\|^{2} + \|b\|^{2} - 2 \la a, b\ra = \frac{1}{2}\|a\|^{2} - \|b\|^{2} + \frac{1}{2}\|a - 2b\|^{2} \ge \frac{1}{2}\|a\|^{2} - \|b\|^{2},\]
we obtain that 
\[\frac{1}{4n}\sum_{i=1}^{n}\|\nabla f_{i}(x^{*})\|^{2} - \frac{1}{2n}\sum_{i=1}^{n}\|\nabla f_{i}(x)\|^{2}\le \frac{1}{2n}\sum_{i=1}^{n}\|\nabla f_{i}(x^{*}) - \nabla f_{i}(x)\|^{2}\le L(f(x) - f(x^{*})).\]
\end{proof}

\begin{proof}[\textbf{Proposition \ref{prop:pairwise}}]
In this case, the RHS of \eqref{eq:general_f} is a form of order-two V-statistics (\cite{vandervaart}) which can be written as 
\[\H \le \frac{2}{m^{2}}\sum_{j,k=1}^{n}U_{jk}.\]
it holds that 
\[\E \left\{\frac{2}{m^{2}}\sum_{j,k=1}^{m}U_{jk}\right\} = O(1),\]
 and 
\begin{align*}
& \Var \left\{\frac{2}{m^{2}}\sum_{j,k=1}^{m}U_{jk}\right\}   = O\lb \frac{1}{m^{4}}\sum_{j,k,l=1}^{m}\Cov(U_{jk}, U_{kl})\rb\\
& = O\lb \frac{1}{m^{4}}\sum_{j,k,l=1}^{m}\sqrt{\Var(U_{jk})\Var(U_{kl})}\rb\\
& = O\lb \frac{1}{m^{4}}\rb = O\lb\frac{1}{m^{4}}\cdot m^{3}\rb = O\lb\frac{1}{m}\rb = o(1).
\end{align*}
The above inequalities imply that $\H = O_{p}(1)$ in this case. The conclusion can be easily extended to the case where $f(x)$ can be written as a higher-order V-statistics. 
\end{proof}

\begin{proposition}
Denote by $x$ the concatenation of $x_{1}, \ldots, x_{K - 1}$ as in Section 5. For multi-class logistic regression loss 
  \[\rho(y_{i}; a_{i}^{T}x) = \log\lb 1 + \sum_{k=1}^{K - 1}e^{a_{i}^{T}x_{k}}\rb - \sum_{k=1}^{K - 1}I(y_{i} = k)a_{i}^{T}x_{k},\]
it holds that 
\[M_{1}\triangleq \sup_{x}\rho_{2}(y_{i};a_{i}^{T}x)^{2}\le 2.\]
\end{proposition}
\begin{proof}
For any $k = 1, \ldots, K - 1$,
\[\pd{f_{i}(x)}{x_{k}} = \lb\frac{e^{a_{i}^{T}x_{k}}}{1 + \sum_{k=1}^{K - 1}e^{a_{i}^{T}x_{k}}} - I(y_{i} = k)\rb\cdot a_{i}\]
and thus
\[\nabla f_{i}(x) = H_{i}(x) \otimes a_{i}\Longrightarrow \|\nabla f_{i}(x)\|^{2} = \|H_{i}(x)\|^{2}\cdot \|a_{i}\|^{2},\]
where
\[H_{i}(x) = \lb \frac{e^{a_{i}^{T}x_{1}}}{1 + \sum_{k=1}^{K - 1}e^{a_{i}^{T}x_{k}}} - I(y_{i} = 1), \ldots, \frac{e^{a_{i}^{T}x_{K - 1}}}{1 + \sum_{k=1}^{K - 1}e^{a_{i}^{T}x_{k}}} - I(y_{i} = K - 1)\rb^{T}.\]
It is easy to see that for any $i$ and $x$
\[\|H_{i}(x)\|^{2}\le \|H_{i}(x)\|_{1}\le 2.\]
This entails that 
\[ \H \le \frac{2}{n}\sum_{i=1}^{n}\|a_{i}\|^{2}.\]
  
\end{proof}

\end{document}